\documentclass{ws-m3as}

\usepackage{graphicx}
\usepackage{latexsym}
\usepackage{amsmath}
\usepackage{enumerate}
\usepackage{amssymb}
\usepackage[mathscr]{eucal}
\usepackage{array}

\renewcommand{\theequation}{\arabic{section}.\arabic{equation}}

\newtheorem{thm}{Theorem}[section]

\newtheorem{lem}[thm]{Lemma}

\newtheorem{alg}[thm]{Algorithm}

\newtheorem{ass}[thm]{Assumption}  %NEW!!!L

\newtheorem{defn}[thm]{Definition}
\newtheorem{rem}[thm]{Remark}
\numberwithin{equation}{section}
\renewcommand{\d}{\mathrm{d}}

\newcommand{\R}{\mathbb{R}}

\newcommand{\Y}{\mathbb{Y}}
\newcommand{\V}{\mathbb{V}}
\newcommand{\Vj}{\mathcal{V}}
\newcommand{\Proba}{\mathbb{P}}
\newcommand {\E} { {\mathbb E} }
\newcommand{\Sph}{\mathbb{S}}
\newcommand{\eps}{\epsilon}
\newcommand{\ph}{\varphi}
\newcommand{\e}{{\rm e}}
\newcommand{\M}{\mathcal M}

\newcommand{\dps}{\displaystyle}

\newcommand{\system}[1]{\left\{ \begin{alignedat}{2}#1\end{alignedat}\right.}

\newcommand{\emat}{\end{pmatrix}}
\newcommand{\abs}[1]{\left | #1\right |}
\newcommand{\set}[1]{\left\{#1\right\}}
\newcommand{\pare}[1]{ \left(#1\right) }
\newcommand{\norm}[1]{\left\Vert#1\right\Vert}

\DeclareMathOperator{\Id}{{\rm Id}}

\DeclareMathOperator{\var}{{\rm var}}
\renewcommand{\div}{ {\rm div}}

\newcommand{\disc}[1]{#1^{\delta t}}
\newcommand{\one}{{\bf 1 }}
\newcommand{\bigo}{\mathcal{O}}
\newcommand{\smallo}{o}

\newcommand{\Max}{\mathcal M}
\newcommand{\bt}{{\bar{t}}}

\newcommand{\taueps}{{\tau_{\eps}}}
\newcommand{\Feps}{{F_{\eps}}}
\newcommand{\dtini}{\overline{\delta t}_{ri}}

\newcommand{\nl}{ {\rm nl} }
\newcommand{\nlbis}{ {\rm nl_{2}} }

\usepackage{todonotes}

\begin{document}

\title{Simulating individual-based models of bacterial chemotaxis with asymptotic variance reduction}
          %For each author, make a block with the following four macros:

\author{Mathias Rousset\thanks{SIMPAF, INRIA Lille - Nord Europe, Lille, France, (mathias.rousset@inria.fr).}  
\;\;\and Giovanni Samaey\thanks{Department of Computer Science, K.U. Leuven,
Celestijnenlaan 200A, 3001 Leuven, Belgium, (giovanni.samaey@cs.kuleuven.be).}}

\pagestyle{myheadings} \markboth{Variance reduced simulation of chemotaxis}{Rousset and Samaey}\maketitle

\begin{abstract}
We discuss variance reduced simulations for an individual-based model of chemotaxis of bacteria with internal dynamics. The variance reduction is achieved via a coupling of this model with a simpler process in which the internal dynamics has been replaced by a direct gradient sensing of the chemoattractants concentrations. In the companion paper \cite{limits}, we have rigorously shown, using a pathwise probabilistic technique, that both processes converge towards the same advection-diffusion process in the diffusive asymptotics. In this work, a direct coupling is achieved between paths of individual bacteria simulated by both models, by using the same sets of random numbers in both simulations. This coupling is used to construct a hybrid scheme with reduced variance. We first compute a deterministic solution of the kinetic density description of the direct gradient sensing model; the deviations due to the presence of internal dynamics are then evaluated via the coupled individual-based simulations. We show that the resulting variance reduction is \emph{asymptotic}, in the sense that, in the diffusive asymptotics, the difference between the two processes has a variance which vanishes according to the small parameter.  
\end{abstract}

% ----------------------------------------------------------------

\section{Introduction}

The motion of flagellated bacteria can be modeled as a sequence of run phases, during which a bacterium moves in a straight line at constant speed. Between two run phases, the bacterium changes direction in a tumble phase, which takes much less time than the run phase and acts as a reorientation. To bias movement towards regions with high concentration of chemoattractant, the bacterium adjusts its turning rate by increasing, resp.~decreasing, the probability of tumbling when moving in an unfavorable, resp.~favorable, direction \cite{Alt:1980p8992,Stock:1999p8984}. Since many species are unable to sense chemoattractant gradients reliably due to their small size, this adjustment is often done via an intracellular mechanism that allows the bacterium to retain information on the history of the chemoattractant concentrations along its path \cite{Bren:2000p7499}. The resulting model, which will be called the ``internal state''  or ``fine-scale'' model in this text, can be formulated as a velocity-jump process, combined with an ordinary differential equation (ODE) that describes the evolution of an internal state that incorporates this memory effect \cite{Erban:2005p4247}. 

% The precise velocity jump model with internal state analyzed in the present paper can be informally described as follows (see Section~\ref{sec:model-internal} for a precise definition). The position of a single bacterium is denoted $t \mapsto X_t \in \R^d$, and the internal variable of the bacterium $t \mapsto Y_t\in \R^n$. $x \mapsto S(x) \in \R^n$ is a space function of the same dimension as the internal variable, and encoding all the necessary information on the chemoattractant concentrations. Then the instantaneous probability rate of occurrence of a tumble phase verifies:
% \[
% \lambda(X_t,Y_t) = \lambda_0 - b . (S(X_t)-Y_t) + \bigo( \abs{S(X_t)-Y_t}^k  ),
% \]
% where $k \geq 2$. In this context, the small parameter $\eps > 0$ is the ratio between the tumbling rate and the speed of the bacterium (a reference length being given by the chemoattractants gradient); see the companion paper~\cite{limits} for a rigorous dimension analysis. The internal variable then follows an ODE of the form:
% \[
%  \frac{ d Y_t}{ dt} = - \taueps ^{-1} (Y_t-S(X_t))+ \text{non-linear terms}, 
% \]
% where the linear part of the ODE is given by the matrix $\taueps ^{-1}$, which is of order $\taueps ^{-1} \sim \eps^{1-\delta}$ with $\delta > 1/k$.

The probability density distribution of the velocity-jump process evolves according to a kinetic equation, in which the internal variables appear as additional dimensions.  A direct deterministic simulation of this equation is therefore prohibitively expensive, and one needs to resort to a stochastic particle method. Unfortunately, a direct fine-scale simulation using stochastic particles presents a large statistical variance, even in the diffusive asymptotic regime, where the behavior of the bacterial density is known explicitly to satisfy a Keller-Segel advection-diffusion equation. Consequently, it is difficult to assess accurately how the solutions of the fine-scale model differ from their advection-diffusion limit in intermediate regimes. We refer to \cite{Horst1,Horst2,KelSeg70} for numerous historical references on the Keller-Segel equation, to \cite{ErbOthm04,Erban:2005p4247} for formal derivations (based on moment closures) of convergence of the the velocity-jump process with internal state to the Keller-Segel equation, and to the companion paper \cite{limits} for a proof of this convergence based on probabilistic arguments.  

In this paper, we propose and analyze a numerical method to simulate  individual-based models for chemotaxis of bacteria with internal dynamics with reduced variance. The variance reduction is based on a coupling technique (control variate): the main idea is to simultaneously simulate, using the same random numbers, a simpler, ``coarse'' process where the internal dynamics is replaced by a direct ``gradient sensing'' mechanism (see \cite{Alt:1980p7984,Othmer:1988p7986,Patlak:1953p7738} for references on such gradient sensing models). The probability density of the latter satisfies a kinetic equation without the additional dimensions of the internal state, and converges to a similar advection-diffusion limit, see e.g.~\cite{Chalub:2004p7641,Hillen:2000p4751,Othmer:2002p4752,limits}.

More precisely, we consider the coupling of two velocity-jump processes with exactly the same advection-diffusion limit (for both processes, a proof of the latter convergence has been obtained in the companion paper~\cite{limits}). The first one is the process with internal dynamics described by the system of equations~\eqref{eq:process_noscale}, with jump rates satisfying \eqref{eq:lin_rate} and internal dynamics satisfying the two main assumptions \eqref{eq:tau_eps_def} and \eqref{eq:tau_scales} as well as other technical assumptions contained in Section~\ref{sec:asympt}. The second one is the process with direct gradient sensing~\eqref{eq:cprocess}, with jump rates satisfying~\eqref{eq:control_rate} with~\eqref{eq:Adef}. The same random numbers are used for the two processes in the definition of both the jump times and the velocity directions.  The main contributions of the present paper are then twofold~:
\begin{itemize}
\item From a numerical point of view, we couple two systems of many particles consisting of different realizations of the fine-scale (internal state) and coarse (gradient sensing) processes, simulated with similar discretization schemes (see Section~\ref{sec:coupl_princ}) and the same random numbers. Additionally, for the coarse model, also the continuum description is simulated on a spatial grid. Evolution of the fine-scale model is then evaluated on the grid at hand by adding the difference between the two coupled particle descriptions to the evolution of the coarse continuum description. This can be seen as using the coarse process as a control variate. As coupling is lost over time, a regular reinitialization of the two coupled particle systems is included in the algorithm. The asymptotic variance reduction method can be depicted using Table~\ref{tab:table} (see also Figure~\ref{fig:pic} in Section~\ref{sec:method}).
% which depicts the principle of the algorithm with the reinitialization procedure): 
\begin{table}\label{tab:table}
\begin{center}
  \begin{tabular}{|l|l|}
\hline
Model & Numerics \\
\hline
Internal state model--Individual description & Monte-Carlo method \\
Gradient sensing model--Individual description & Monte-Carlo method (coupled)\\
Gradient sensing model--Density description & Grid method \\
%Advection-Diffusion model--Density description & Grid method \\
\hline
\end{tabular}
\end{center}
\caption{Models and numerical methods used in the variance reduction algorithm.}
\end{table}
	\item From an analysis point of view, we rigorously prove, using probabilistic arguments, $L^p$-bounds on the position difference at diffusive times of both processes in terms of appropriate powers of the scaling parameter (Theorem~\ref{thm:coupling}). This analysis shows that the numerical variance reduction obtained by the method is \emph{asymptotic} (the remaining statistical variance vanishes with the small scaling parameter), and provides upper bounds on the rate of the latter convergence. We argue in Section~\ref{sec:sharp} that this bound should be sharp in relevant regimes.
\end{itemize}

The idea of asymptotic variance reduction is a general and very recent idea in scientific computing that appears when using hybrid Monte-Carlo/PDE (partial differential equation) methods. While the use of control variates is fairly common in Monte Carlo simulation, the only explicit attempt to develop asymptotic variance reduction in hybrid methods, up to our knowledge, can be found in \cite{DimPar08} and related papers in the context of the Boltzmann equation, and is based on an importance sampling approach (see also~\cite{10.1063/1.1899210} and related papers). 

This paper is organized as follows. In Section \ref{sec:model}, we briefly review the mathematical models that we will consider, and summarize the results on their advection-diffusion limit that were obtained in the companion paper \cite{limits}. In Section \ref{sec:method}, we proceed to describe the coupled numerical method in detail. 
Section \ref{sec:variance} is dedicated to a detailed analysis of the variance reduction of the coupled method. The main result is the following: up to a time shift perturbation, the variance of the difference between the two coupled process scales according to the small parameter of the diffusive asymptotics.
We proceed with numerical illustrations in Sections \ref{sec:illustr} and \ref{sec:appl}, and conclude in Section \ref{sec:concl} with an outlook to future research.

\section{Particle-based models for bacterial chemotaxis \label{sec:model}}

In this paper, we directly present the models of interest in a nondimensional form, incorporating the appropriate space and time scales.  For a  discussion on the chosen scaling, we refer to the companion paper \cite{limits}.

\subsection{Model with internal state\label{sec:model-internal}}
We consider bacteria that are sensitive to the concentration of $m$ chemoattractants $\pare{\rho_i(x)}_{i=1}^m$, with $\rho_i(x) \geq 0$ for $x \in \R^d$. While we do not consider time dependence of chemoattractant via production or consumption  by the bacteria, a generalization to this situation is straightforward, at least for the definition of the models and the numerical method.  Bacteria move with a constant speed $v$ (run), and change direction at random instances in time (tumble), in an attempt to move towards regions with high chemoattractant concentrations.  As in \cite{ErbOthm04}, we describe this behavior by a velocity-jump process driven by some internal state $y \in \mathbb{Y}\subset\R^{n}$ of each individual bacterium. The internal state models the memory of the bacterium and is subject to an evolution mechanism attracted by a function  $\psi: \R^m \to \R^n$ of the chemoattractants concentrations,
\[
S(x):= \psi(\rho_1(x),\ldots,\rho_m(x)) \in \R^n,
\]
where $x$ is the present position of the bacterium. $S$ is assumed to be smooth with bounded derivatives up to order~$2$.
%A possible typical situation is $2m = n$, $\psi(\rho_1, \ldots,\rho_m) = (\rho_1, \ldots,\rho_m,0, \ldots, 0)$, and the internal mechanism being such that the $m$ first internal variables are memorizing the past concentrations of chemoattractants, while the $m$ last internal variables are computing the variations of concentrations along a bacterium trajectory (see Section~2.2 in~\cite{limits}).
 We refer to Section~2.2 in~\cite{limits} for some typical choices; in the concrete example in this paper, we choose $m=n=1$, and $S(x)=\rho_1(x)$. By convention, the gradient of $S(x):\R^d \to \R^n$ is a matrix with dimension
\[
\nabla S(x) \in \R^{n\times d}.
\]
We then denote the evolution of each individual bacterium position by~
$
  t \mapsto X_t,
$
with normalized velocity
\[
  \frac{\d X_t}{\d t} = \epsilon V_t, \qquad V_t\in \mathbb{V}=\Sph^{d-1},
\]                             
with $\Sph^{d-1}$ the unit sphere in $\R^d$. Hence, $V_t$ represents the direction and the parameter $\epsilon$ represents the size of the velocity.  The evolution of the internal state is denoted by
$t \mapsto Y_t$. The internal state adapts to the local chemoattractant concentration through an ordinary differential equation (ODE),
\begin{equation}\label{eq:internal}
  \frac{\d Y_t}{\d t}= F_\eps(Y_t,S(X_t)),
\end{equation}
which is required to have a unique fixed point $y^*=S(x^*)$ for every fixed value $x^* \in \R^d$. 
%% Isn't F now proportional to eps^(1-delta), see later on ?  In any case, that is what the first paper now says.
%In the above, the function $F$ is normalized so that the largest time scale of the convergence of $Y_t$ to $y^*$ is of order $\bigo(1)$ when imposing any value $X_t\equiv x^*$ fixed in time.  
We also introduce the deviations from equilibrium $z=S(x)-y$.  The evolution of these deviations is denoted as 
\[
t \mapsto Z_t = S(X_t) -Y_t.
\]
The velocity of each bacterium is switched at random jump times
 $(T_n)_{n \geq 1}$
that are generated via a Poisson process with a time dependent rate given by $\lambda(Z_t)$, where $z \mapsto \lambda(z)$ is a smooth function satisfying
\begin{equation}
  \label{eq:ratebound}
  0 < \lambda_{\rm min} \leq \lambda \pare{z} \leq \lambda_{\rm max},
\end{equation}
as well as,
\begin{equation}\label{eq:lin_rate}
\lambda(z) = \lambda_0 - b^T  z + c_\lambda\bigo \pare{ \abs{z}^k } ,
\end{equation}
with $b \in \R ^ n $, $k \geq 2$, and $c_\lambda >0$ is used to keep track of the non-linearity in the analysis. The new velocity at time $T_n$ is generated at random according to a centered probability distribution
$\M (dv)$ 
with $\int v \, \M (dv) =0 $, typically
\[
\M (dv) = \sigma_{\Sph^{d-1}}(dv),
\]
where $\sigma_{\Sph^{d-1}}$ is the uniform distribution on the unit sphere.

The resulting fine-scale stochastic evolution of a bacterium is then described by a left continuous with right limits (lcrl) process
 \[
 t \mapsto \pare{X_t, V_t, Y_t} ,
 \] 
which satisfies the following differential velocity-jump equation~:
\begin{equation} \label{eq:process_noscale}
  \begin{cases}
   \dps \dfrac{\d X_t}{\d t} = \epsilon  V_t,  \\[8pt]
    \dps \dfrac{\d Y_t}{\d t} =  F_\eps(Y_t,S(X_t)), \\[8pt]
    \dps \int_{T_n}^{T_{n+1}} \lambda(Z_t) \d t = \theta_{n+1}, \qquad \text {with }
    Z_t := S(X_t)-Y_t ,\\[8pt]
    \dps V_{t}  =  \Vj_{n} \quad \text{ for $t \in [T_{n},T_{n+1}]$ } , 
  \end{cases}
\end{equation}
with initial condition $X_0, \Vj_0 \in \R^d$, $Y_0 \in \R^n$ and $T_0 = 0$. 
In \eqref{eq:process_noscale}, $\pare{\theta _n }_{n \geq 1}$ denote i.i.d. random variables with normalized exponential distribution, and $\pare{ \Vj_n }_{n \geq 1}$ denote i.i.d. random variables with distribution $ \mathcal M (dv)$.

\begin{example}
For concreteness, we provide a specific example, adapted from \cite{ErbOthm04}, which will also be used later on to illustrate our results numerically. We consider $m=1$, i.e., there is only one chemoattractant $S(x)$ and the internal dynamics~\eqref{eq:internal} reduces to a scalar equation
\begin{equation}\label{e:scalar-y}
	\frac{\d y(t)}{\d t} = \frac{S(x)-y}{\tau}=\frac{z}{\tau}.
\end{equation}
For the turning rate $z\mapsto \lambda(z)$, we choose the following nonlinear strictly decreasing smooth function 
\begin{equation}\label{eq:rate}
\lambda(z)=2\lambda_0 \pare{\frac{1}{2}-\frac{1}{\pi}\arctan \pare{\frac{\pi}{2\lambda_0}z}}.
\end{equation}
\end{example}

\subsection{Asymptotic regimes and linearizations\label{sec:asympt}}

In the adimensional models~\eqref{eq:process_noscale}-\eqref{eq:cprocess}, a small parameter $\eps$ is introduced, which can be interpreted as the ratio of the typical time of random change of velocity direction to  the typical time associated with chemoattractant variations as seen by the bacteria.  In the present section, we present the main scaling assumptions on the internal dynamics that are necessary to consider the asymptotic regime~$\eps \ll 1$, as well as the corresponding notation.  Most of these assumptions were introduced in the companion paper~\cite{limits}, to which we refer for more details and motivation.

\begin{rem}[Notation]
Throughout the text, the Landau symbol $\bigo$ denotes a \emph{deterministic} and globally Lipschitz function satisfying $\bigo(0)=0$. Its precise value \emph{may vary} from line to line, may depend on all parameters of the model, but its Lipschitz constant is uniform in $\eps$. In the same way, we will denote generically by $C > 0$ a deterministic constant that may depend on all the parameters of the model except $\eps$.
\end{rem}

We assume that the ODE~\eqref{eq:internal} driving the internal state is well approximated by a near equilibrium evolution equation in following sense~:
\begin{ass}\label{a:1}
 We have
\begin{equation}\label{eq:tau_eps_def}
 \Feps(y,s) =  - \taueps^{-1} (y-s) + \eps^{1-\delta}c_F\bigo(\abs{s-y}^2),
\end{equation} 
where $\delta >0$ and $\taueps  \in\R^{n\times n}$ is an invertible constant matrix.
\end{ass}
The constant $c_F$ is used to keep track of the dependance on the non-linearity of $F_\eps$ in the analysis. For technical reasons that are specific to the analysis of the time shift in the coupling (Lemma~\ref{lem:DTestimbis}), we need the following assumption on the scale of the non-linearity (this assumption was not necessary in the companion paper~\cite{limits}):
\begin{ass}\label{a:1bis}
 Assume that $\abs{y-s} = \bigo(\eps^\delta)$, then there is a $\gamma > 0$ such that:
\begin{equation}\label{eq:taueps_hypbis}
\abs { (y-s) - \taueps\Feps(y,s) } =  c_F\bigo( \eps^{1 + \gamma} ).
\end{equation} 
\end{ass}

We now formulate the necessary assumptions on the scale of the matrix $\taueps$~:
\begin{ass}\label{a:2}
There is a constant $C >0$ such that for any $t \geq 0$, one has
\begin{equation}\label{eq:tau_scales}
  \norm{\e^{- t \taueps ^{-1} }} \leq C \e^{- t \eps^{1-\delta}/C}.
\end{equation}
\end{ass}
\begin{ass}\label{a:3}
There is a constant $C >0$ such that
\[
 \sup_{t \geq 0} \norm{t\taueps ^{-1}\e^{- t \taueps ^{-1} } } \leq C  \eps^{-1}.
\]
\end{ass}
Assumption~\ref{a:2} is used to ensure exponential convergence of the linear ODE with time scale at least of order $\eps^{1-\delta}$. Assumption~\ref{a:3} is necessary for technical reasons in the proof of Lemma~\ref{lem:DTestim} (which is given in~\cite{limits}). Note that if $\taueps^{-1}$ is symmetric and strictly positive with a lower bound of order $\bigo(\eps^{1-\delta})$ on the spectrum, then Assumption~\ref{a:2} and Assumption~\ref{a:3} are satisfied. This implies that the slowest time scale of the ODE~\eqref{eq:internal} is at worse of order $\eps^{1-\delta}$.  Indeed, integration over time in Assumption~\ref{a:2} implies $\norm{\taueps} \leq C \eps^{\delta-1}$. 
Moreover, we  assume that the solution of the ODE driving the internal state in \eqref{eq:process_noscale} satisfies the following long time behavior~:
\begin{ass}
   \label{ass:ODE}
Consider a path $t \mapsto S_t$ such that $ \sup_{t \in [0,+\infty] } \abs{ \dfrac{ d S_t}{ dt} }= \bigo( \eps )$, and assume $\abs{Y_0 - S_0} = \bigo(\eps^\delta)$. Then the solution of
\[
\frac{ d Y_t}{ dt} = \Feps(Y_t,s_t),
\]
satisfies
\[
\sup_{t \in [0,+\infty] } \abs{ Y_t - S_t } = \bigo \pare{ \eps^\delta  }.
\]
\end{ass}
Assumption~\ref{ass:ODE} is motivated by the case of linear ODEs; indeed, in that case, Assumption~\ref{ass:ODE} is a consequence of Assumption~\ref{a:2} (see Section~3.3 in~\cite{limits}). 
Recalling that $Z_t=S_t-Y_t$, we will assume in the remainder of the paper that the solution of~\eqref{eq:process_noscale} satisfies,
\[
\abs{Z_0} = \bigo(\eps^\delta), 
\]
as well as 
\[
\abs{\taueps^{-1} Z_0} = \bigo( 1 ).
\]
Then, under Assumption~\ref{a:1},  Assumption~\ref{a:2},  Assumption~\ref{a:3}, and Assumption~\ref{ass:ODE}, we obtain the bounds (see Section~3.3 in~\cite{limits})
\begin{equation}
    \label{eq:techbound}
    \sup_{t \in [0,+\infty] } \abs{ \taueps^{-1}\e^{- \taueps^{-1} t } (S_t - Y_t) } = \bigo \pare{ 1 } ,
  \end{equation}
\[
\sup_{t \in [0,+\infty]} \pare{ \abs{Z_t} } =  \bigo(\eps^{\delta}),
\]
as well as
\[
\sup_{t \in [0,+\infty]} \pare{ \abs{ \taueps^{-1} \e^{- \taueps^{-1} t} Z_t} } =  \bigo(1) .
\]

Finally, we will make use of the following notation. Denoting 
$c_S = \norm{\nabla^{2}  S}_{\infty},$
with $\nabla^2$ the Hessian, 
the error terms due to non-linearities of the problem will be handled through
\begin{equation}
  \label{eq:nl_err}
  \begin{cases}
     \nl(\eps) := c_{F} \eps^{\delta} + c_S \eps + c_{\lambda} \eps^{ k \delta-1}  = \smallo(1),\\
     \nlbis(\eps) := c_{F}( \eps^{\delta} + \eps^\gamma)+ c_S( \eps + \eps^\delta) + c_{\lambda} \eps^{ k \delta-1}  = \smallo(1) .
  \end{cases}
\end{equation}
 Note that $ \nl(\eps) \leq \nlbis(\eps)$ and $ \nl(\eps) \sim \nlbis(\eps)$ if $\delta \leq 1$ and $\delta \leq \gamma$. Recall that $k \geq 2$ is defined in~\eqref{eq:lin_rate}.

Note that all the assumptions of the present section hold in the following case: (i) the ODE is linear ($F_\eps$ is linear), (ii) the associated matrix $\tau_{\eps}$ is symmetric positive and satisfies
\[
\tau_{\eps}^{-1} \geq C \eps^{1-\delta},
\]
for some $\delta > 0$, (iii) the initial internal state is close to equilibrium in the sense that $\taueps^{-1} Z_0=\bigo(1)$. The assumptions of the present section can be seen as technical generalizations to non-linear ODEs with non-symmetric linear part.

\subsection{Model with direct gradient sensing (control process)}

% IDENTICAL --> SHORTEN ?
We now turn to a simplified, ``coarse'' model, in which the internal process (\ref{eq:internal}), and the corresponding state variables, are eliminated.  Instead, the turning rate  depends directly on the chemoattractant gradient. This process will be called the \emph{control process} since it will be used in Section~\ref{sec:method} as a control variate to perform variance reduced simulations of (\ref{eq:process_noscale}).

The control process is a Markov process in position-velocity variables
\[
t \mapsto (X^c_t,V^c_t) ,
\]
which evolves according to the following differential velocity-jump equations~:
\begin{equation} \label{eq:cprocess}
\system{
& \frac{\d X_t^c}{\d t} = \eps V_t^c , &\\
& \int_{T_n^c}^{T_{n+1}^c} \lambda^c_\eps(X_t^c , V_t^c) \d t = \theta_{n+1} , & \\
& V_{t}^c  =  \Vj_n  \quad \text{ for $t \in [T_{n}^c,T^c_{n+1}]$ } , & 
}
\end{equation}
with initial condition $X_0, \Vj_0 \in \R^d$.  In (\ref{eq:cprocess}), $\pare{\theta _n }_{n \geq 1}$ denote i.i.d. random variables with normalized exponential distribution, and $\pare{ \Vj_n }_{n \geq 1}$ denote i.i.d. random variables with distribution $ \Max (dv)$.
The turning rate of the control process is assumed to satisfy
\begin{equation}
  \label{eq:rateboundc}
  0 < \lambda_{\rm min} \leq \lambda^c_{\eps}(x,v) \leq \lambda_{\rm max},
\end{equation}
as well as
\begin{equation}\label{eq:control_rate}
\lambda^c_\eps(x,v) := \lambda_0 - \eps\; A^T_\eps(x)  v  + \bigo(\eps^2).
\end{equation}

Typically, $\lambda^c_{\eps}(x,v)$ is a function of $\nabla S(x)$, so that the model \eqref{eq:control_rate} may describe a large bacterium that is able to directly sense chemoattractant gradients. When $m=n=1$, and the turning rate (\ref{eq:control_rate}) is proportional to $ \nabla S(x)  v \in \R$, it can be interpreted as follows: the rate at which a bacterium will change its velocity direction depends on the alignment of the velocity with the gradient of the chemoattractant concentration $\nabla S(x)$, resulting in a transport towards areas with higher chemoattractant concentrations.  
% The model can also be derived from the internal state model \eqref{eq:process_noscale} in the special asymptotics where bacteria learn the value of the chemoattractants concentrations infinitely fast \cite{limits}. 
In the companion paper \cite{limits}, it is shown that, when the vector field $A_\eps(x) \in \R^{d}$ is given by the formula
\begin{equation}
  \label{eq:Adef}
  A_\eps(x) = b^T \frac{\taueps}{\Id + \lambda_0 \taueps} \nabla S(x),
\end{equation}
asymptotic consistency with the internal state model is obtained in the limit when $\eps\to 0$ (see Section~\ref{sec:model-hydro}). 

\subsection{Kinetic formulation and advection-diffusion limits\label{sec:model-hydro}}

We now turn to the kinetic description of the probability distribution of the processes introduced above.
The probability distribution density of the fine-scale process with internal state at time $t$ with respect to the measure $\d x \, \Max ( \d v ) \, \d y$ is denoted as
$p(x,v,y,t)$, suppressing the dependence on $\eps$ for notational convenience, and evolves according to the Kolomogorov forward evolution equation (or master equation). In the present context, the latter is the following kinetic equation
\begin{equation} \label{e:kinetic}
\partial_t p + \eps v \cdot \nabla_x p + \div_y  \pare{ F_\eps(x,y) p } = \lambda \pare{S(x)-y} \pare{   R(p)    - p   },
\end{equation}
where 
\[
R(p) := \int_{v \in \Sph^{d-1}} p(\cdot,v,\cdot) \, \Max(dv)
\]
is the operator integrating velocities with respect to $\Max$, and $F(x,y)$ and $\lambda(z)$ are defined as in Section \ref{sec:model-internal}. 
Similarly, the distribution density of the control process is denoted as
$p^c(x,v,t)$, and evolves according to
the kinetic/master equation:
\begin{equation} \label{e:kinetic_control}
\partial_t p^c + \eps v \cdot \nabla_x p^c = \pare{   R(\lambda^c_\eps p)    - \lambda^c_\eps p   },
\end{equation}
with $\lambda^c_\eps$ defined as in \eqref{eq:control_rate}.
The reader is referred to \cite{EthKur86} for the derivation of master equations associated to Markov jump processes.

In the companion paper \cite{limits}, we show, using probabilistic arguments, that, in the limit $\eps \to 0$, both the equation for the control process~\eqref{e:kinetic_control} and the equation for the process with internal state~\eqref{e:kinetic} converge to an advection-diffusion limit on diffusive time scales.  Convergence is to be understood \emph{pathwise}, in the sense of convergence of probability distribution on paths endowed with the uniform topology. 

Denote
diffusive times by 
\[
\bt:=t  \eps^2,
\]
and the processes considered on diffusive time scales as
\[
X^{\eps}_\bt:=X_{\bt/\eps^2}, \qquad X^{c,\eps}_\bt:=X^c_{\bt/\eps^2}.
\]
For the control process, we then have the following theorem ~:
\begin{theorem}
For $\eps \to 0$, the process $\bt \mapsto X^{c,\eps}_\bt$, solution of~\eqref{eq:cprocess}, converges towards an advection-diffusion process, satisfying the stochastic differential equation (SDE)
\begin{equation}\label{eq:cprocess_hydro}
\d X_{\bar{t}}^{c,0} =\pare{\frac{D A_0(X_{\bar{t}}^{c,0})}{\lambda_0} \d \bt + \pare{\frac{2D}{\lambda_0}}^{1/2} \d W_{\bar{t}}},
\end{equation}
where $\bt \mapsto  W_\bt$ is a standard Brownian motion, the parameters $\lambda_0$ and $A_0:=\lim_{\eps\to 0} A_\eps$ originate from the turning rate~\eqref{eq:control_rate}, and the diffusion matrix is given by the covariance of the Maxwellian distribution~:
\begin{equation}
  \label{eq:D}
  D =  \int_{\Sph^{d-1}} v \otimes v \, \Max(dv) \in \R^{d \times d}.
\end{equation}
\end{theorem}
In particular, this result implies that, at the level of the Kolomogorov/master evolution equation, the evolution of the position bacterial density,
\begin{equation}\label{eq:cdensity}
	n^{c,\eps}(x,\bt):= n^c(x,\bt/\eps^2):=\int_{\V}p^c(x,v,\bt/\eps^2) \,  \Max(\d v) 
\end{equation}
converges to the advection-diffusion equation \begin{equation}\label{eq:cdens_hydro}
\partial_\bt n^{c,0} = \frac{1}{\lambda_0}\div_x \pare{ D \nabla_x n^{c,0} - D A_0(x) n^{c,0}}
\end{equation}
on diffusive time scales as $\epsilon \to 0$.

In the same way, a standard probabilistic diffusion approximation argument can be used to derive the pathwise diffusive limit of the process with internal state~\eqref{eq:process_noscale}~:
\begin{theorem}
For $\eps \to 0$, the process $\bt \mapsto X^{\eps}_\bt$, solution of~\eqref{eq:process_noscale}, converges towards an advection-diffusion process, satisfying the stochastic differential equation (SDE)~\eqref{eq:cprocess_hydro}, where $A_0$ originates from 
\begin{equation}\label{eq:control_field}
A_0(x) = b^T \lim_{\eps \to 0}  \frac{\taueps}{\lambda_0 \taueps + \Id } \nabla S(x),
\end{equation}
in which, $b$, $\taueps$, and $\lambda_0$ were introduced in \eqref{eq:tau_eps_def}-\eqref{eq:lin_rate} as parameters of the process with internal state, and $\Id \in \R^{n\times n}$ is the identity matrix. Again, the diffusion matrix $D$ is given by the covariance of the Maxwellian distribution~\eqref{eq:D}.
\end{theorem}
Also here, convergence needs to be understood in terms of probability distribution on paths endowed with the uniform topology.

Introducing the bacterial density of the process with internal state as
\begin{equation}
	n(x,t)= \int_{\Y}\int_{\V}p(x,v,y,t) \Max(\d v) \d y,
\end{equation}
this implies that the evolution of $n$ converges to \eqref{eq:cdens_hydro} on diffusive time scales in the limit of $\epsilon \to 0$. 

\section{Numerical method\label{sec:method}}

To simulate the process with internal state, solving the kinetic equation \eqref{e:kinetic} over diffusive time scales can be cumbersome, due to the additional dimensions associated with the internal state. The alternative is to use to stochastic particles. However, a particle-based simulation of equation \eqref{e:kinetic}  is subject to a large statistical variance of the order $\mathcal{O}(N^{-1/2})$, where $N$ is the number of simulated particles. The asymptotic analysis shows that the position bacterial density approaches an advection-diffusion limit \eqref{eq:cdens_hydro} when $\eps\to 0$. Consequently, to accurately assess the deviations of the process with internal state \eqref{eq:process_noscale} as compared to its advection-diffusion limit (for small but non-vanishing (intermediate) values of $\eps$), the required number of particles needs to increase substantially with decreasing $\eps$, which may become prohibitive from a computational point of view.

In this section, we therefore propose a hybrid method, based on the principle of control variates, that couples the process with internal dynamics with the control process, which is simulated simultaneously using a grid-based method. We first describe the variance reduction technique (Section \ref{sec:coupl_princ}).  The analysis in Section~\ref{sec:variance} will reveal that the variance reduction is \emph{asymptotic}, in the sense that the variance vanishes in the diffusion limit.  To ensure this asymptotic variance reduction during actual simulations, one needs to ensure that the time discretization preserves the diffusion limits of the time-continuous process.  An appropriate time discretization is discussed in Section \ref{sec:discr}.

\subsection{Coupling and asymptotic variance reduction\label{sec:coupl_princ}}

The proposed variance reduction technique is based on the introduction of a control variate that exploits a coupling between the process with internal state and the control process.
While the idea of control variates for Monte Carlo simulation is already well known, see e.g.~\cite{KloPla92} and references therein, the coupling that is proposed here is particular (we call it \emph{asymptotic}), since the difference of the coupled processes, and hence the variance, vanishes with an estimable rate in the diffusion limit $\epsilon\to 0$, as will be shown in Section \ref{sec:variance}. 

\subsubsection{The control variates}
Let us first assume that we are able to compute the exact solution of the kinetic equation for the control process, \eqref{e:kinetic_control}, with infinite precision in space and time.

The algorithm of asymptotic variance reduction is based on a coupling between an ensemble of realizations evolving according to the process with internal state \eqref{eq:process_noscale}, denoted as
\[ \left\{X^i_t,V^i_t,Y^i_t\right\}_{i=1}^N,
\]
and an ensemble of realization of the control process \eqref{eq:cprocess}, denoted as 
\[ \left\{X^{i,c}_t,V^{i,c}_{t}\right\}_{i=1}^N.
\]
We denote the empirical measure of the particles with internal state in position-velocity space as 
\[
\mu_\bt^{N}(x,v) = \frac{1}{N} \sum_{i=1}^{N} \delta_{X^i_{\bt/\eps^2},V^i_{\bt/\eps^2}},
\]
and, correspondingly, the empirical measure of the control particles as
\[
\mu_\bt^{c,N} = \frac{1}{N} \sum_{i=1}^{N} \delta_{X^{c,i}_{\bt/\eps^2},V^{c,i}_{\bt/\eps^2}}.
\]

A coupling between the two ensembles is obtained by ensuring that both simulations use \emph{the same random numbers} $(\theta_n)_{n\ge 1}$ and $(\Vj_n)_{n\ge 0}$, which results in a strong correlation between $(X^i_{t},V^{i}_t)$ and $(X_t^{i,c},V^{c,i}_t)$ for each realization.
Simultaneously, the kinetic equation for the control process \eqref{e:kinetic_control} is also solved using a deterministic method (which, for now, is assumed to be exact). We formally denote the corresponding semi-group evolution as 
\[
\e^{\bt L^c}, \qquad \text{ with } L^c(p^c)= - \eps v \cdot \nabla_x p^c + \pare{   R(\lambda^c_\eps p^c)    - \lambda^c_\eps p^c   }.
\]
Besides the two particle measures $\mu^N_{\bar{t}}$ and $\mu^{c,N}_{\bar{t}}$, we denote by $\overline{\mu}^N_{\bar{t}}$ the variance reduced measure, which will be defined by the algorithm below. Since, with increasing diffusive time, the variance of the algorithm increases due to a loss of coupling between the particles with internal state and the control particles, the variance reduced algorithm will also make use of a reinitialization time step $\dtini$, which is defined on the diffusive time scale. The corresponding time instances are denoted as $\bar{t}_n=n\dtini$ on the diffusive time scale, or equivalently, on the original time scale as $t_n=n\dtini/\eps^2$.

Starting from an initial probability measure $\mu_0$ at time $t=0$, we sample $\mu_0$ to obtain the ensemble $\left\{X^i_t,V^i_t,Y^i_t\right\}_{i=1}^N$, corresponding to $\mu^N_0$, and then set $\mu^{c,N}_0:=\mu^N_0$, i.e., $X^{i,c}_{0}=X^{i}_{0}$ and $V^{i,c}_{0}=V^{i}_{0}$ for all $i=1,\ldots,N$.  Furthermore, we set the variance reduced estimator as $\overline{\mu}_0^N:=\mu_0=\E(\mu_0^N)$. We then use the following algorithm to advance from $\bar{t}_n$ to $\bar{t}_{n+1}$, (see also Figure~\eqref{fig:pic})~:
\begin{figure}
	%\begin{center}
  \includegraphics[scale=0.5]{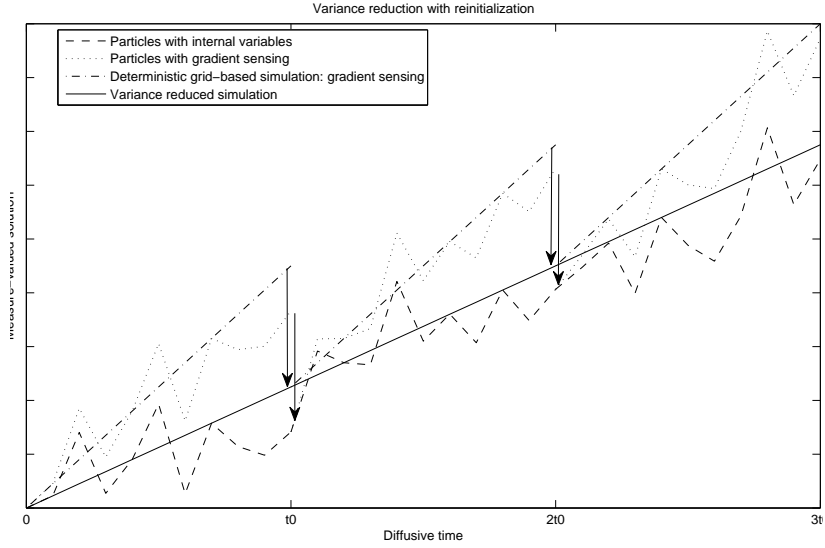}
	\caption{\label{fig:pic} A schematic description of Algorithm \ref{algo}. The dashed line represent the evolution of $N$ bacteria with internal state. The dotted line represent the coupled evolution of $N$ bacteria with gradient sensing, subject to regular reinitializations. The dashed-dotted line is computed according to a deterministic method simulating the density of the model with gradient sensing, and subject to regular reinitializations. The solid line is the variance reduced simulation of the internal state dynamics, and is computed by adding the difference between the particle computation with internal state, and the particle simulation with gradient sensing to the deterministic gradient sensing simulation.  At each reinitialization step, the two simulations (deterministic and particles) of the gradient sensing dynamics are reinitialized to the values of their internal state simulation counterpart (as represented by the arrows).}
%\end{center}
\end{figure}
\begin{alg}\label{algo}
At time $t_n$, we have that the particle measure $\mu_{\bar{t}_n}^{c,N}= \mu_{\bar{t}_n}^{N}$, and the variance reduced measure is given by  $\overline{\mu}^N_{\bar{t}_n}$.  
To advance from time $ \bar{t}_n$ to $\bar{t}_{n+1}$, we perform the following steps~:
  \begin{itemize}
  \item Evolve the particles $\left\{X^i_t,V^i_t,Y^i_t\right\}_{i=1}^N$ from $t_n$ to $t_{n+1}$, according to \eqref{eq:process_noscale},
   \item  Evolve the particles $\left\{X^{i,c}_t,V^{i,c}_{t}\right\}_{i=1}^N$ according to \eqref{eq:cprocess}, using the same random numbers as for the process with internal state,
 \item Compute the variance reduced evolution
   \begin{equation}
     \label{eq:var_reduced_estimation}
     \overline{\mu}^N_{\bar{t}_{n+1}} = \overline{\mu}^N_{\bar{t}_n} \e^{\dtini / \eps^2 L^c}  + \pare{\mu_{\bar{t}_{n+1}}^{N} - \mu_{\bar{t}^{-}_{n+1}}^{c,N} }.
   \end{equation}
\item Reinitialize the control particles by setting 
\[
X^{i,c}_{t_{n+1}}=X^{i}_{t_{n+1}}, \qquad V^{i,c}_{t_{n+1}}=V^{i}_{t_{n+1}}, \qquad i=1,\ldots,N,
\]
i.e., we set the state of the control particles to be identical to the state of the particles with internal state. 
\end{itemize}
\end{alg}
In formula \eqref{eq:var_reduced_estimation}, we use the symbol $\bar{t}^{-}_{n+1}$ to emphasize that the involved particle positions and velocities are those obtained \emph{before} the reinitialization.
An easy computation shows that the algorithm is unbiased in the sense that for any $n \geq 0$,
\[
\E \pare{ \overline{\mu}^N_{\bar{t}_n}} = \E \pare{\mu^N_{\bar{t}_n}},
\]
since the particles with internal dynamics are unaffected by the reinitialization, and, additionally, 
\[
\E \mu^{c,N}_{\bar{t}_{n+1}}= \E \overline{\mu}^N_{\bt_n}\e^{\dtini/ \eps^2 L^c}. 
\]
Moreover, the variance is controlled by the coupling between the two processes. Indeed, using the independence of the random numbers between two steps of Algorithm~\ref{algo}, and introducing $\ph$ as a position and velocity dependent test function), we get (according to the main result of Theorem~\ref{thm:coupling})
\begin{align}
  \var( \overline{\mu}^N_{\bar{t}_n}(\ph)  ) &= \sum_{k=1}^n \E \pare{\mu_{}^{N}(\ph) - \mu_{\bar{t}_k^{-}}^{c,N}(\ph) }^2 \nonumber \\
 & \leq  \sum_{k=1}^n \frac{\norm{\nabla \ph}_\infty}{N} \E \pare{  \abs{X_{\bar{t}_k/\eps^2} - X_{\bar{t}_k/\eps^2}^c  }^2 } , \nonumber
\end{align}
and thus
 \begin{align}\label{eq:MC_bound}
  \var( \overline{\mu}^N_{\bar{t}_n}(\ph)  ) 
& \leq C n\frac{ \eps + \eps^\delta + \nlbis(\eps) }{N},
\end{align}
where in the last line, $C$ is independant of $n$, $\eps$, and $N$.

In some generic situations (see Section~\ref{sec:sharp}), we can argue that the statistical error in Algorithm~\ref{algo} coming from the coupling is ``sharp'' with respect to the order in $\eps$. This means that the difference between the probability distribution of the model with internal state and the probability distribution of the model with gradient sensing is of the same order.  This would imply that, with the asymptotic variance reduction technique, one is able to reliably assess the true deviation of the process with internal variables from the control process using a number of particles $N$ that is independent of $\eps$.

\subsubsection{Time-space discretization of the kinetic equation}

For both processes, the probability density distribution of the position and velocity can then be computed via binning in a histogram, or via standard kernel density estimation \cite{WScott:1992p8096,WSilverman:1986p8151} using a kernel~$K_h$, in which the chosen bandwidth $h$ can be based on the grid used for the deterministic simulation and on the data, for instance using the Silverman heuristic \cite{WSilverman:1986p8151}. This yields
\begin{equation}\label{eq:kde}
	\hat{p}_N(x,v,t) = \frac{1}{N h} \sum_{i=1}^{N} K_h(x - X^i_{t},v-V^i_{t}),
\end{equation}
and similarly 
\begin{equation}\label{eq:kdec}
	\hat{p}^c_N(x,v,t) = \frac{1}{N h} \sum_{i=1}^{N} K_h(x - X^{c,i}_{t},v-V^{c,i}_{t}).
\end{equation}

The important point is that the solution $p^c(x,v,t)$ for the control process may be accurately approximated with a deterministic grid-based method, which ensures that 
\[
\abs{p^c(x,v,\bt / \eps^2)-p^c_{\rm grid}(x,v,\bt / \eps^2)} = \bigo(\delta x^l) +\bigo(\delta t^k),
\]
for some integers $k,l\ge 1$.
\begin{remark}
In this text, we will perform the simulations in one space dimension using a third-order upwind-biased scheme, and perform time integration using the standard fourth order Runge--Kutta method, i.e., $l=3$ and $k=4$.
\end{remark} 
Then, the unbiased nature of the variance-reduced estimator is conserved up to $\bigo(\delta x^l)+\bigo(\delta t^k)+\bigo(h)$ discretization errors.

\subsection{Time discretization of velocity-jump processes \label{sec:discr}}

When simulating equation~\eqref{eq:process_noscale}, a time discretization error originates from the fact that the equation for the evolution of the internal state, and hence the evolution of the fluctuations $Z_t$, is discretized in time.  This results in an approximation of the jump times $(T_n)_{n\ge 1}$, and hence of $X_t$. To retain the diffusion limits of the time-continuous process, special care is needed. We now briefly recall the time discretization procedure that was proposed and analyzed in the companion paper~\cite{limits}. For ease of exposition, we consider the scalar equation \eqref{e:scalar-y} for the internal state; generalization to nonlinear systems of equations is briefly discussed in~\cite{limits}. 

Consider first the linear turning rate \eqref{eq:lin_rate}. In that case, we define a numerical solution  $(\disc{X}_t,\disc{V}_t,\disc{Y}_t)$ as follows.
Between jumps, we discretize the simulation in steps of size $\delta t$ and denote by $(\disc{X}_{n,k},\disc{Z}_{n,k})$ the solution at $t_{n,k}=\disc{T}_n+k\delta t$.  The numerical solution for $t\in[t_{n,k},t_{n,k+1}]$ is given by 
\begin{equation} \label{e:discr-nonlin}
  \begin{cases}
    \disc{X}_t = \disc{X}_{n,k}+\eps \Vj_{n}\;(t-t_{n,k})  \\
    \disc{Z}_t = \exp(-(t-t_{n,k}) \taueps^{-1}) \disc{Z}_{n,k}  + \eps \taueps \pare{\Id-\exp(-(t-t_{n,k})\taueps^{-1})} \nabla S(\disc{X}_{n,k}) \, \Vj_n.
  \end{cases}
\end{equation}
We denote by $K\geq 0$ the integer such that the simulated jump time $\disc{T}_{n+1}\in [t_{n,K},t_{n,K+1}]$. To find $\disc{T}_{n+1}$, we first approximate the integral $ \int_{\disc{T}_n}^{\disc{T}_{n+1}}\lambda(\disc{Z}_t)\d t$ using 
\begin{align}\label{eq:time-disc-lin}
	\int_{\disc{T}_n}^{\disc{T}_{n+1}}\lambda(\disc{Z}_t)\d t &=\sum_{k=0}^{K-1}\int_{t_{n,k}}^{t_{n,k+1}}\lambda(\disc{Z}_t)\d t 
	+ \int_{t_{n,K}}^{\disc{T}_{n+1}}\lambda(\disc{Z}_t)\d t ,
\end{align}
and then compute~:
\begin{multline}
	\int_{t_{n,k}}^{t_{n,k+1}}\lambda(\disc{Z}_t)\d t= 
	\lambda_0 \delta t - b^T \pare{\Id - \e^{- {\delta t}{\taueps^{-1}}} }{\taueps} \disc{Z}_{T_n}  \\
 - \eps b^T \pare{\delta t \taueps -(\Id-\e^{-\delta t \taueps^{-1} } )\taueps^2   }  \nabla S(\disc{X}_{n,k}) \Vj_n .
\end{multline}
The jump time $\disc{T}_{n+1}$ can then be computed as the solution of 
\begin{equation}\label{e:lin-Newton}
\int_{t_{n,K}}^{\disc{T}_{n+1}}\lambda(\disc{Z}_t)\d t=\theta_{n+1}- \sum_{k=0}^{K-1}\int_{t_{n,k}}^{t_{n,k+1}}\lambda(\disc{Z}_t)\d t,
\end{equation}
using a Newton procedure. It is shown in \cite{limits} that the results on the diffusion limit (as outlined in Section \ref{sec:model-hydro}) are not affected by the discretization. 

We now consider a general nonlinear turning rate. We again discretize in time to obtain the time-discrete solution \eqref{e:discr-nonlin}.  The jump time $\disc{T}_{n+1}$ is now computed by linearizing \eqref{eq:rate} in each time step, 
\begin{equation}\label{eq:nonlinrate_disc}
\disc{\lambda}\pare{\disc{Z}_t,\disc{Z}_{n,k}}=\lambda(\disc{Z}_{n,k})+ \frac{\d \lambda(\disc{Z}_{n,k})}{\d z}\pare{\disc{Z}_t-\disc{Z}_{n,k}}.
\end{equation}
One then replaces equation \eqref{eq:time-disc-lin} by 
\begin{align}
	\int_{\disc{T}_n}^{\disc{T}_{n+1}}\disc{\lambda}(\disc{Z}_t)\d t &=\sum_{k=0}^{K-1}\int_{t_{n,k}}^{t_{n,k+1}}\disc{\lambda}(\disc{Z}_t,\disc{Z}_{n,k})\d t 
	+ \int_{t_{n,K}}^{\disc{T}_{n+1}}\disc{\lambda}(\disc{Z}_t,\disc{Z}_{n,K})\d t,  
\end{align}
and proceeds in the same way as for the linear case.  Also in this case, the diffusive limit is recovered in an exact fashion for the time-discretized process.

\begin{remark}[Discretization of the control process]
For the control process, there is no internal state. Hence, the only time dependence of turning rate $\lambda^c$ is due to the spatial variation of $\nabla S(x)$, which can be treated by discretizing the integral of $\lambda$ in the same way as above.
\end{remark}

\section{Asymptotic variance reduction of the coupling\label{sec:variance}}

In this section, we show how the difference between the two coupled processes on diffusive time scales $\bar{t}=t/\eps^2$ behaves in the limit of $\eps\to 0$.  We first recall some notation from the companion paper~\cite{limits} (Section~\ref{sec:not}), after which we state and prove the main theorem (Section~\ref{sec:proof}).

\subsection{Notations and asymptotic estimates\label{sec:not}}

The main theorem in Section~\ref{sec:proof} relies on following auxiliary definitions and lemmas that were given in the companion paper~\cite{limits}:
\begin{defn}
  We denote by $m:\R \to \R^{n \times n}$ the function
  \begin{equation}
    \label{eq:m(t)}
    m(t) := t \taueps - \pare{\Id - {\rm e}^{-t \taueps^{-1}}} \taueps^2,
  \end{equation}
whose derivative is given by
\[
m'(t) = \taueps\pare{\Id - {\rm e}^{-t \taueps^{-1}} }.
\]
\end{defn}
This function satisfies the following lemmas~:
\begin{lem}
  Let $\theta$ be an exponential random variable of mean $1$. Then:
  \begin{equation}
    \label{eq:av_m}
    \E \pare{ m\left( \frac{\theta}{\lambda_0} \right)  } = \frac{1}{\lambda_0} \frac{\taueps}{\Id + \lambda_0 \taueps}=A_\eps(x),
  \end{equation}
see equation \eqref{eq:Adef}.
\end{lem}
\begin{lem}
For all $t \in \R$, we have
$\norm{ m'(t) } \leq t$, as well as
$\norm{ m(t) } \leq t^2/2$.
\end{lem}

The difference between jump times is denoted as
\begin{equation*}
\Delta T_{n+1}^{c} :=  T_{n+1}^{c}- T_{n}^{c}, \qquad \Delta T_{n+1} :=  T_{n+1}- T_{n}.
\end{equation*}

The proofs of the asymptotic variance reduction make use of the following asymptotic expansions of the jump time differences of both processes~:
\begin{lem}\label{lem:dt-control}
The difference between two jump times of the control process satisfies
\begin{equation}\label{eq:deltaTc}
\Delta T_{n+1}^{c} = \frac{\theta_{n+1}}{\lambda_0} + \eps \frac{\theta_{n+1}}{\lambda_0^2} A_\eps^T(X_{T^c_n}^c) \Vj_n + \theta_{n+1} \bigo (\eps^2  ).
\end{equation}
\end{lem}
and
\begin{lem}\label{lem:DTestim} 
The jump time variations of the process with internal state can be written in the following form ($\nl(\eps)$ being defined by~\eqref{eq:nl_err})~:
\begin{equation}\label{eq:deltaT}
  \Delta T_{n+1} = \Delta T_{n+1}^0 + \eps \Delta T_{n+1}^1 + (\theta_{n+1}^6+\theta_{n+1})\bigo \pare{\eps^2+ \eps\,\nl(\eps)} ,
\end{equation}
where 
\begin{align}\label{eq:estimDT0}
\Delta T_{n+1}^0 &= \frac{\theta_{n+1}}{\lambda_0} + \frac{b^T}{\lambda_0} m'(\Delta T_{n+1}^0 ) Z_{T_n} \nonumber \\   
&= \frac{\theta_{n+1}}{\lambda_0} + \theta_{n+1} \bigo ( \eps^\delta) 
\end{align}
and, correspondingly, 
\begin{align}\label{eq:estimDT1}
  \Delta T_{n+1}^1 &= \frac{1}{\lambda_0 - b^T \e^{- {\Delta T_{n+1}^0}{\tau^{-1}}} Z_{T_n} } b^T m(\Delta T_{n+1}^0) \nabla  S(X_{T_n})  \Vj_{n}\nonumber  \\ 
& = \frac{1}{\lambda_0}   b^T m\pare{\frac{\theta_{n+1}}{\lambda_0}} \nabla S(X_{T_n}) \Vj_{n} +  \theta_{n+1} \bigo(\eps^\delta) .
\end{align}
\end{lem}
It is also useful to recall that according to \eqref{eq:ratebound}-\eqref{eq:rateboundc}, the following hold~:
\[
\abs{\Delta T_{n+1}} \leq C \theta_{n+1}, \qquad \abs{\Delta T_{n+1} ^c} \leq C \theta_{n+1}.
\]
Finally, we will also need a different expansion of the jump times of the process with internal state~:
\begin{lem}\label{lem:DTestimbis}
When using $\nl_2(\eps)$ as defined by~\eqref{eq:nl_err}, the jump time variations of the process with internal state can be written in the following form~:
\begin{equation}\label{eq:deltaTbis}
  \begin{split}
&     \Delta T_{n+1} = \frac{\theta_{n+1}}{\lambda_0} - \frac{b^T \taueps}{\lambda_0}(Z_{T_{n+1}}-Z_{T_{n}})+ \eps \Delta T _{0 } b^T \taueps \nabla S (X_{T_n}) \Vj_n \\
& + \pare{\theta_{n+1}^3+\theta_{n+1}}\bigo \pare{ \eps\,\nlbis(\eps) + \eps^{1+\delta} }.
  \end{split}
 \end{equation}
\end{lem}
\begin{proof}
  The proof is based on Duhamel integration of the ODE~\eqref{eq:process_noscale} on $[T_n,T_{n+1}]$, as in the proof of Lemma 4.7 of the companion paper~\cite{limits}. Following equation~(4.24) in~\cite{limits}, we get for $t\in [T_n,T_{n+1}]$~:
\begin{align}\label{eq:key_Duhamel}
  Z_t = \e^{- \pare{t-T_n}{\taueps^{-1}}} Z_{T_n} + \eps \pare{\Id - \e^{- \pare{t-T_n}{\taueps^{-1}}} }{\taueps} \nabla S(X_{T_n})  \Vj_{n} + (\theta_{n+1}^2+\theta_{n+1})\bigo \pare{c_S \eps^2+ c_F \eps^{1+\delta} }.
\end{align}
Multiplying both sides by $\tau_\eps$, recalling that Assumption~\ref{a:2} implies $\|\tau_\eps\|=\bigo(\eps^{\delta-1})$, and using Assumption~\ref{a:1bis}, we get for $t\in [T_n,T_{n+1}]$~:
\begin{align}\label{e:help}
 \taueps Z_t = \taueps \e^{- \pare{t-T_n}{\taueps^{-1}}} Z_{T_n} + \eps \pare{\Id - \e^{- \pare{t-T_n}{\taueps^{-1}}} }{\taueps^2} \nabla S(X_{T_n})  \Vj_{n} + (\theta_{n+1}^2+\theta_{n+1})\bigo \pare{c_S \eps^{1+\delta}+ c_F \eps^{1+\gamma} }.
\end{align}
Moreover, integrating~\eqref{eq:key_Duhamel} on $t\in [T_n,T_{n+1}]$ yields
\begin{align}\label{eq:key_Duhamelint}
  \begin{split}
   \int_{T_n}^{T_{n+1}} Z_t \d t &= \pare{\Id - \e^{- \pare{T_{n+1}-T_n}{\taueps^{-1}}} }{\taueps} \pare{Z_{T_n} -\eps\taueps \nabla S(X_{T_n}) \Vj_{n} } \\
& \quad +\eps (T_{n+1}-T_n)\taueps \nabla S(X_{T_n})  \Vj_{n} + (\theta_{n+1}^3+\theta_{n+1}^2)\bigo \pare{c_S \eps^2+ c_F \eps^{1+\delta} }. 
\end{split}
\end{align}
Evaluating equation~\eqref{e:help} at $t=T_{n+1}$, add adding equation~\eqref{eq:key_Duhamelint}, we get
\begin{align*}
  \int_{T_n}^{T_{n+1}} Z_t \d t + \taueps Z_{T_{n+1}}  = & \taueps Z_{T_n} +\eps (T_{n+1}-T_n)\taueps \nabla S(X_{T_n})  \Vj_{n}  \\
& \quad  + \pare{\theta_{n+1}^3+\theta_{n+1}}\bigo \pare{c_S(\eps^{1+\delta}+\eps^2 ) + c_F(\eps^{1+\delta}+\eps^{1+\gamma}) } . 
\end{align*}
The estimates in Lemma~$4.8$ from~\cite{limits} (i.e., $\abs{ \Delta T_{n+1}-\Delta T_{n+1}^0} = (\theta_{n+1}^4+\theta_{n+1})\bigo(\eps) $) yields
\begin{align*}
  \int_{T_n}^{T_{n+1}} Z_t \d t + \taueps Z_{T_{n+1}}  = & \taueps Z_{T_n} +\eps \Delta T_{n+1}^0 \taueps \nabla S(X_{T_n})  \Vj_{n}  \\
& \quad  + \pare{\theta_{n+1}^3+\theta_{n+1}}\bigo \pare{c_S(\eps^{1+\delta}+\eps^2 ) + (c_F+1)\eps^{1+\delta}+c_F\eps^{1+\gamma}) } . 
\end{align*}

Finally, plugging the last equation in the estimate~\eqref{eq:lin_rate} yields the result.
% \begin{align}\label{eq:DTsuite}
%   \theta_{n+1} & = \lambda_0 \Delta T_{n+1} - b^T  \int_{T_n}^{T_{n+1}} Z_t dt +\theta_{n+1} \bigo \pare{\eps^{2\delta}} \nonumber \\
%  &=  \lambda_0 \Delta T_{n+1} + b^T \taueps \pare{ Z_{T_{n+1}} - Z_{T_n}} - \eps \Delta T_{n+1}b^T\taueps \nabla S(X_{T_n})  \Vj_{n} \nonumber \\
% & \qquad + (\theta_{n+1}^3+1) \pare{\eps^{2\delta} }.
% \end{align}
% Finally, the estimate~(4.27) from~\cite{limits} (, i.e., $\abs{ \Delta T_{n+1}-\Delta T_{n+1}^0} = (1+\theta_{n+1}^3)\bigo(\eps) $), and $\eps \norm{\taueps} = \bigo(\eps^\delta)$, yield the result.
\end{proof}

\subsection{Analysis of variance of the coupling\label{sec:proof}}
The present section is devoted to the proof of the following theorem:
\begin{thm}\label{thm:coupling}
Assume the assumptions of Section~\ref{sec:asympt} hold, and that $k \delta > 1$ where $k\geq 2$ is defined in~\eqref{eq:lin_rate}. Then the difference between the process with internal state and the coupling process satisfies~:
\begin{equation}
  \label{eq:coupl_estim}
   \E \pare{\pare { X_{\bt/\eps^2} - X^c_{\bt/\eps^2}  } ^ p}^{1/p} = \bigo( \eps + \eps^\delta + \nlbis(\eps)  ), \qquad p\ge 1,
\end{equation}
where $\nlbis(\eps)$ is defined in~\eqref{eq:nl_err}.
\end{thm}

The proof of Theorem~\ref{thm:coupling} relies on a number of steps that will be detailed by a series of lemmas. We will make use of the following random sequences, defined as the position of both processes after $n$ jumps,
\begin{equation}
  \label{eq:Xi}
  \Xi_n := X_{T_n}= X_0 + \eps \sum_{m=0}^{n-1}\Delta T _{m+1} \Vj_m, \qquad n\geq 0,
\end{equation}
as well as
\begin{equation}
  \label{eq:Xic}
  \Xi^c_n := X_{T^c_n}^c= X_0 + \eps \sum_{m=0}^{n-1}\Delta T _{m+1}^c \Vj_m, \qquad n\geq 0.
\end{equation}
We  will also make use of the following random integers~:
\begin{defn}\label{def:N_Nc}
  The random integers $N \geq 1$ and $N^c \geq 1$ are uniquely defined by:
\begin{equation}
  T_N \leq \bt / \eps^2 < T_{N+1}, \qquad T^c_{N^c} \leq \bt / \eps^2 < T^c_{N^c+1} .
\end{equation}
\end{defn}

The first lemma bounds the difference at time $\bt / \eps^2$ between the two processes $X_{\bt/\eps^2}$ and $X^c_{\bt/\eps^2}$ by expressing it in terms of  differences of positions and jump times of both processes after the same random number of jumps, $\Xi_n$ and $\Xi_n^c$~:
\begin{lem}
  The difference between the rescaled process with internal state $X^\eps_{\bt}:= X_{\bt/\eps^2}$ and the rescaled coupling process $X^{c,\eps}_{\bt}:= X^c_{\bt/\eps^2}$ satisfies:
  \begin{align}
\abs{X^\eps_{\bt} - X^{\eps,c}_{\bt}} & \leq  \abs{\Xi_{N}-\Xi_{N}^c} + \eps \abs{T_{N}-T_{N}^c}+\pare{\theta_{N+1}+\theta_{N_c+1}} \bigo(\eps) \label{eq:coupl_11} \\
&  \leq \abs{\Xi_{N^c}-\Xi_{N^c}^c} + \eps \abs{T_{N^c}-T_{N^c}^c}+\pare{\theta_{N+1}+\theta_{N_c+1}} \bigo(\eps) \label{eq:coupl_12} 
\end{align}
\end{lem}
Only one of the two estimates \eqref{eq:coupl_11}-\eqref{eq:coupl_12} is necessary in the remainder of the proof, but we detail both to highlight the symmetry.
\begin{proof}
  By definition, we have
  \begin{equation*}
    \begin{cases}
      X^\eps_{\bt}= \Xi_N+ \eps(\bt/\eps^2-T_N)\Vj_{N}, \\
      X^{c,\eps}_{\bt}= \Xi_{N^c}^c+ \eps(\bt/\eps^2-T_{N^c}^c)\Vj_{N},
    \end{cases}
  \end{equation*}
so that by Definition~\ref{def:N_Nc} of $N$ and $N^c$, and by realizing that $\abs{\bt/\eps^2-T_{N^c}^c} \leq C \theta_{N_c+1}$ and $\abs{\bt/\eps^2-T_{N}} \leq C \theta_{N+1}$,  we get
\begin{equation}\label{eq:diffX}
    \begin{cases}
     \abs{ X^\eps_{\bt}- X^{c,\eps}_{\bt}}\leq \abs{\Xi_N-\Xi^c_N}+\underbrace{\abs{\Xi^c_{N^c}-\Xi^c_{N}}}_{(a)}+\pare{\theta_{N+1}+\theta_{N_c+1}} \bigo(\eps)  , \\
     \abs{ X^\eps_{\bt}- X^{c,\eps}_{\bt}} \leq \abs{\Xi_{N^c}-\Xi^c_{N^c}}+\underbrace{\abs{\Xi_{N^c}-\Xi_{N}}}_{(b)}+ \pare{\theta_{N+1}+\theta_{N_c+1}} \bigo(\eps).
    \end{cases}
  \end{equation}
To analyze $(a)$, we consider three cases:
\begin{itemize}
\item[$N=N_c$.] Then $(a)=0$.
\item[$N > N_c$.] Then 
  \begin{align*}
    \frac{(a)}{\eps} & = \abs{\sum_{n = N_c}^{N-1} \Delta T _{n+1}^c \Vj_n } \\
& \leq T _{N}^c-T_{N_c}^c=\Delta T^c_{N^c+1} + T^c_ N - T^c_{N^c+1}  \\
&\leq T^c_ N - T_{N} + C \theta_{N^c+1},
  \end{align*}
where in the last line we have used that, $\Delta T^c_{N^c+1}\le C\theta_{N^c+1}$, and, by definition, $ T_{N} \leq \bt/ \eps^2 < T^c_{N^c+1}$.
\item[$N_c > N$.] Then
  \begin{align*}
    \frac{(a)}{\eps} & = \abs{\sum_{n = N}^{N_c-1} \Delta T _{n+1}^c \Vj_n }  \\
& \leq T _{N_c}^c-T_{N}^c \\
&\leq T_{N+1} - T_{N}^c  = \Delta T_{N+1} + T_N-T_N^c \leq T_ N - T^c_{N} + C \theta_{N+1} ,
  \end{align*}
where in the last line we have used that by definition
 $ T^c_{N^c} \leq \bt/ \eps^2 \leq T_{N+1}$.
\end{itemize}
This proves~\eqref{eq:coupl_11}. By symmetry, we get:
\[
\frac{(b)}{\eps} \leq \abs{T_{N^c} - T^c_{N^c}} + C\pare{\theta_{N+1}+\theta_{N^c+1}},
\]
which yields~\eqref{eq:coupl_12}.
\end{proof}

In the next lemma, we estimate the supremum of the difference of the position of both processes after $n$ jumps for $n\in [0,n_\eps]$.
\begin{lem}
  Let $n_\eps \geq 1$ be a deterministic integer verifying $n_\eps = \bigo(\eps^{-2})$. Then 
  \begin{equation}
    \label{eq:estim_jumps}
    \E \pare{ \sup_{0 \leq n \leq n_\eps}\abs{\Xi_n-\Xi_n^c}^p } ^{1/p} = \bigo(\eps + \eps^\delta + \nl(\eps) ). 
  \end{equation}
\end{lem}
\begin{proof}
  First, using the estimates of jump times~\eqref{eq:deltaTc}-\eqref{eq:deltaT}, we can decompose the differences of positions of both processes as follows~:
  \begin{align*}
    \Xi_{n}-\Xi_{n}^c  &= \sum_{m=0}^{n-1}\eps (\Delta T _ {m+1}-\Delta T _ {m+1} ^c) \Vj_{m} \\
& = \sum_{m=0}^{n-1} \alpha_{m+1} + d_{m+1} + r_{m+1},
  \end{align*}
where, by definition,
\begin{equation*}
  \begin{cases}
    \alpha_{m+1} := \eps^2 \frac{\Vj_m}{\lambda_0^2} \pare{A_\eps(\Xi_{m})^T-A_\eps(\Xi_{m}^c)^T } \Vj_m, \\
       d_{m+1} :=\eps^2\frac{\Vj_m}{\lambda_0^2} \pare{1-\theta_{m+1}}A_\eps(\Xi_{m}^c)^T  \Vj_m \\
\hspace{1.5cm} + \eps^2 \frac{\Vj_{m}}{\lambda_0^2} \pare{ \lambda_0  b^T m\pare{\frac{\theta_{m+1}}{\lambda_0}} \nabla S(\Xi_m) - A_\eps(\Xi_{m})^T }  \Vj_{m} \\
\hspace{1.5cm} + \eps b^T m'(\Delta T_{m+1}^0 ) Z_{T_m} \Vj_m , \\
r_{m+1}:= \pare{\theta_{m+1}^6+\theta_{n+1}} \bigo(\eps^{3} + \eps^2 \nl(\eps) ).
     \end{cases}
  \end{equation*}
%\todo[inline]{Change symbol for $m$; you already have $m$ as a counter and as the function $m(t)$.  What would be a good choice for the term $m$ ?  If you prefer $m$ here for ``martingale'', maybe change $m(t)$ to something else. DONE. }
%\todo[inline]{Why a $\sup$ for $m$ and not for $r$. BECAUSE FOR r THE SUM OF  abs IS ALREADY THE SUP.}

Since $A_\eps$ is Lipschitz, we have that $\alpha_{m+1} \leq C \eps^2 \abs{\Xi_{m}-\Xi_{m}^c}$, and we can write 
\begin{equation*}
  \abs{\Xi_{n}-\Xi_{n}^c} \leq  C\eps^2 \sum_{m=0}^{n-1} \abs{\Xi_{m}-\Xi_{m}^c} + \sup_{0 \leq l \leq n_{\eps}}\abs{\sum_{m=1}^{l} d_{m}}+ \sum_{m=1}^{n_\eps}\abs{r_{m}},
\end{equation*}
for $n \leq n_{\eps}$, which yields
\begin{align}
  \abs{\Xi_{n}-\Xi_{n}^c} &\leq \pare{\sup_{0 \leq l \leq n_{\eps}}\abs{\sum_{m=1}^{l} d_{m}}+ \sum_{m=1}^{n_\eps}\abs{r_{m}}} \pare{1+C \eps^2+ \dots + (C \eps^2)^{n_\eps}}, \label{eq:key_xi}.
\end{align}
Note that
\[
\pare{1+C \eps^2+ \dots + (C \eps^2)^{n_\eps}} \leq  C .
\]
Now, we can remark that the discrete time process
$
\pare{M_l}_{l \geq 0}:=\pare{ \sum_{m=1}^{l} d_{m} }_{l \geq 0}
$ is a martingale with respect to the filtration $\mathcal F _{l} = \sigma\pare{(\theta_{m},\Vj_{m-1}) \vert 1 \leq m \leq l }$ with $M_0=0$, so that we can apply the Burkholder-Davies-Gundy upper bound (which is a simple consequence of Doob maximal inequality here, see~\cite{Karatzasbook}), for some $p \geq 1$:
\begin{equation}
  \label{eq:BDG}
  \E \pare{\sup_{0 \leq l \leq n_\eps } \abs{M_l}^p} \leq C_p n_{\eps}^{p/2 - 1} \sum_{l=1}^{n_{\eps}} \E \pare{\abs{M_l-M_{l-1}}^p}.
\end{equation}
In the present case, since $\abs{Z_{T_n}}=\bigo(\eps^\delta)$:
\[
\E \pare{\abs{d_l}^p} \leq C_p \eps^{p(1+\delta)}.
\]
Moreover,
\begin{align*}
  \E \pare{\abs{\sum_{m=1}^{n_\eps}\abs{r_{m}}}^p} &\leq n_{\eps}^{p - 1} \sum_{m=1}^{n_\eps}\E \abs{r_{m}}^p \\
& \leq C \pare{ \eps + \nl(\eps) }^p,
\end{align*}
and the result follows from~\eqref{eq:key_xi}.
\end{proof}
In the same way, we can estimate the difference of the $n_\eps$-th jump time of both processes~:
\begin{lem}
  Let $n_\eps \geq 1$ a deterministic integer. Then 
  \begin{equation}
    \label{eq:estim_jumps_T}
    \eps \, \E \pare{ \sup_{0 \leq n \leq n_\eps}\abs{T_n-T_n^c}^p }^{1/p} \leq C_p \pare{ \eps^\delta + \nlbis(\eps)  } 
  \end{equation}
\end{lem}
\begin{proof}
   First, using the estimates of jump times~\eqref{eq:deltaTc}-\eqref{eq:deltaTbis}, we can decompose the differences of jump times of both processes as follows~:
  \begin{align*}
    T_{n}-T_{n}^c  &= \sum_{m=0}^{n-1} \Delta T _ {m+1}-\Delta T _ {m+1} ^c\\
& = \sum_{m=0}^{n-1} \delta_{m+1} + d_{m+1} + r_{m+1},
  \end{align*}
where, by definition,
\begin{equation*}
  \begin{cases}
    \delta_{m+1} := - \frac{b^T \taueps}{\lambda_0}(Z_{T_{m+1}}-Z_{T_{m}}), \\
       d_{m+1} :=\eps \Delta T _{0 } b^T \taueps \nabla S (X_{T_n}) \Vj_n \\
\hspace{1.5cm} - \eps \frac{\theta_{n+1}}{\lambda_0^2} A_\eps^T(X_{T^c_n}^c) \Vj_n  \\
r_{m+1}:= \pare{\theta_{m+1}^3+ \theta_{n+1} } \eps \bigo(\eps^\delta + \nlbis(\eps)).
     \end{cases}
  \end{equation*}
We can write for $n \leq n_{\eps}$:
\begin{align*}
  \abs{T_{n}-T_{n}^c} \leq \abs{\frac{b^T \taueps}{\lambda_0}(Z_{T_{n}}-Z_{T_{0}}) } + \sup_{0 \leq l \leq n_{\eps}}\abs{\sum_{m=1}^{l} d_{m}} + \sum_{m=1}^{n_\eps}\abs{r_{m}}.
\end{align*}
Now, we can remark that the discrete time process
$
\pare{M_l}_{l \geq 0}:=\pare{ \sum_{m=1}^{l} d_{m} }_{l \geq 0}
$ is a martingale with respect to the filtration $\mathcal F _{l} = \sigma\pare{(\theta_{m},\Vj_{m-1}) \vert 1 \leq m \leq l }$ with $M_0=0$. Using the Burkholder-Davies-Gundy inequality~\eqref{eq:BDG}, together with:
\[
\E \pare{\abs{m_{m}}^p} \leq C \eps^{p\delta},
\]
yields
\[
\E \pare{ \sup_{0 \leq l \leq n_{\eps}}\abs{\sum_{m=1}^{l} d_{m}}^p  }^{1/p} \leq C \eps^{\delta-1}.
\]
Using:
\[
\abs{\frac{b^T \taueps}{\lambda_0}(Z_{T_{n}}-Z_{T_{0}}) } = \bigo(\eps^{2 \delta-1}) < \bigo(\eps^{\delta-1}),
\]
and plugging in the rest term, we finally get~\eqref{eq:estim_jumps_T}.
\end{proof}

We can now conclude with the proof of Theorem~\ref{thm:coupling}~:

\begin{proof}[Proof of Theorem~\ref{thm:coupling}]
We consider a given integer $n_{\eps} \geq 1$. We decompose the difference between both processes as 
\begin{equation}\label{eq:proof1}
  \abs{ X^\eps_{\bt}- X^{c,\eps}_{\bt}} = \abs{ X^\eps_{\bt}- X^{c,\eps}_{\bt}} \one_{N\leq n_{\eps}}+\abs{ X^\eps_{\bt}- X^{c,\eps}_{\bt}} \one_{N > n_{\eps}},
\end{equation}
and use~\eqref{eq:coupl_11} to write 
\begin{equation}
  \label{eq:proof2}
  \abs{ X^\eps_{\bt}- X^{c,\eps}_{\bt}} \one_{N\leq n_{\eps}} \leq \sup_{0 \leq n \leq n_\eps}\abs{\Xi_n-\Xi_n^c} +\eps \sup_{0 \leq n \leq n_\eps}\abs{T_n-T_n^c} + \pare{\theta_{N+1}+\theta_{N_c+1}} \bigo(\eps).
\end{equation}
Using now the estimates~\eqref{eq:estim_jumps} and \eqref{eq:estim_jumps_T} in~\eqref{eq:proof2}, we get
\begin{align}\label{eq:estim_jumps_tot}
  \E \pare{ \abs{ X^\eps_{\bt}- X^{c,\eps}_{\bt}}^p \one_{N \leq n_{\eps}} }^{1/p}& \leq C  ( \eps + \eps^\delta + \nlbis(\eps)  )\\
& \qquad + C\eps \E (\theta_{N}^p+\theta_{N^c}^p)^{1/p} \nonumber \\
& = \bigo( \eps + \eps^\delta +  \nlbis(\eps)     ),
\end{align}
where in the last line we have used a technical lemma (Lemma~\ref{lem:tech1}).

It remains to control the probability of the event $\{ N > n_\eps\}$. To this purpose, we consider $t \mapsto N_t \in \mathbb N$ the Poisson process uniquely defined by:
\[
\sum_{n=1}^{N_t} \theta_n \leq t < \sum_{n=1}^{N_t +1} \theta_n.
\]
By definition of $N$ and $N_c$, there is a constant $C$ such that $N_{C\bt / \eps^2} \geq N$ and $N_{C\bt / \eps^2} \geq N^c$, so that we can write:
\begin{equation}
  \label{eq:proof3}
  \abs{ X^\eps_{\bt}- X^{c,\eps}_{\bt}}  \leq \eps C \pare{\sum_{n=1}^{N_{C\bt / \eps^2}+1}\theta_{n}}\one_{N_{C\bt / \eps^2} > n_{\eps}}.
\end{equation}
Let us choose $n_{\eps}  \geq  4 \e C\bt / \eps^2 $, a choice that  satisfies $n_{\eps}=\bigo(\eps^{-2})$. A second technical lemma (Lemma~\ref{lem:tech2}) implies that
\begin{equation}
  \label{eq:proof4}
  \E \pare{ \abs{ X^\eps_{\bt}- X^{c,\eps}_{\bt}}^p \one_{N > n_{\eps}} } \leq \eps C_p 2^{-1/(C\eps^2)},
\end{equation}
so that the contribution of the latter to the final estimate~\eqref{eq:coupl_estim} is negligible, since it is exponentially small with respect to $\eps^{-2}$. The proof is complete.
\end{proof}

\subsection{Discussion about the sharpness of the coupling}\label{sec:sharp}
One may ask about the ``sharpness'' of the precise estimates of the coupling, given mainly by Theorem~\ref{thm:coupling}; (see also~\eqref{eq:estim_jumps} (position shift), and \eqref{eq:estim_jumps_T} (time shift)). Let us discuss the case where $c_F=0$, $ \delta \in [1/k , 1/(k-1)]$. In this case the dominant error term in the jump times expansion of the internal state model in Lemma~$4.8$ in~\cite{limits} is due to the non-linearity of the turning rate: $c_\lambda \eps^{k \delta -1}$. Since this term is due to the internal state mechanism, it can be conjectured that the difference between the density of the internal state model and gradient sensing model will be at least of this order. On the other hand, the coupling estimate in Theorem~\ref{thm:coupling} is controlled by the same term: $c_\lambda \eps^{k \delta -1}$ , suggesting the sharpness of the latter.

\section{Numerical illustration\label{sec:illustr}}

In this section, we  demonstrate the validity of the analysis above. For the numerical experiments, we restrict ourselves to one space dimension.
For the process with internal state, we use internal dynamics that are given by the scalar cartoon model \eqref{e:scalar-y}. In all experiments, $\tau$ is chosen independently of $\eps$ so that $\delta=1$. The corresponding turning rate is given by \eqref{eq:rate} with $\lambda_0=1$. For the control process, we choose the linear turning rate \eqref{eq:control_rate}, with parameters \eqref{eq:control_field} such that the two processes have the same diffusion limit; in particular, $b=1$. The physical domain is $x\in [0,20]$, and we use reflecting boundary conditions, {\it i.e}. the bacterial velocity is reversed when $x=0$ or $x=20$. We fix a scalar bimodal chemoattractant concentration field
\begin{equation}\label{eq:chemoattractant}
	S(x)=\alpha\left(\exp\left(-\beta\left(x-\xi\right)^2\right)+\exp\left(-\beta\left(x-\eta\right)^2\right)\right),
\end{equation}
in which $\alpha=5$, $\beta=1$, $\xi=7.5$ and $\eta=12.5$.

\paragraph{Single bacterium as a function of time. } In a first experiment, we simulate a single bacterium evolving according to the fine-scale model (\ref{eq:process_noscale}), in which we set the parameters to $\eps=0.2$ and $\tau=1$.  We compare this evolution to that of a bacterium that satisfies the corresponding control process \eqref{eq:cprocess}. Both simulations are performed using the same random numbers starting from the initial condition $X_0=X_0^c=8$ and $V_0=V_0^c=+1$; the bacterium with internal state starts from $Y_0=S(X_0)$. The time step $\delta t=0.1$. The results are shown in figure \ref{fig:one-particle}.
\begin{figure}
	\begin{center}
	\includegraphics[scale=0.8]{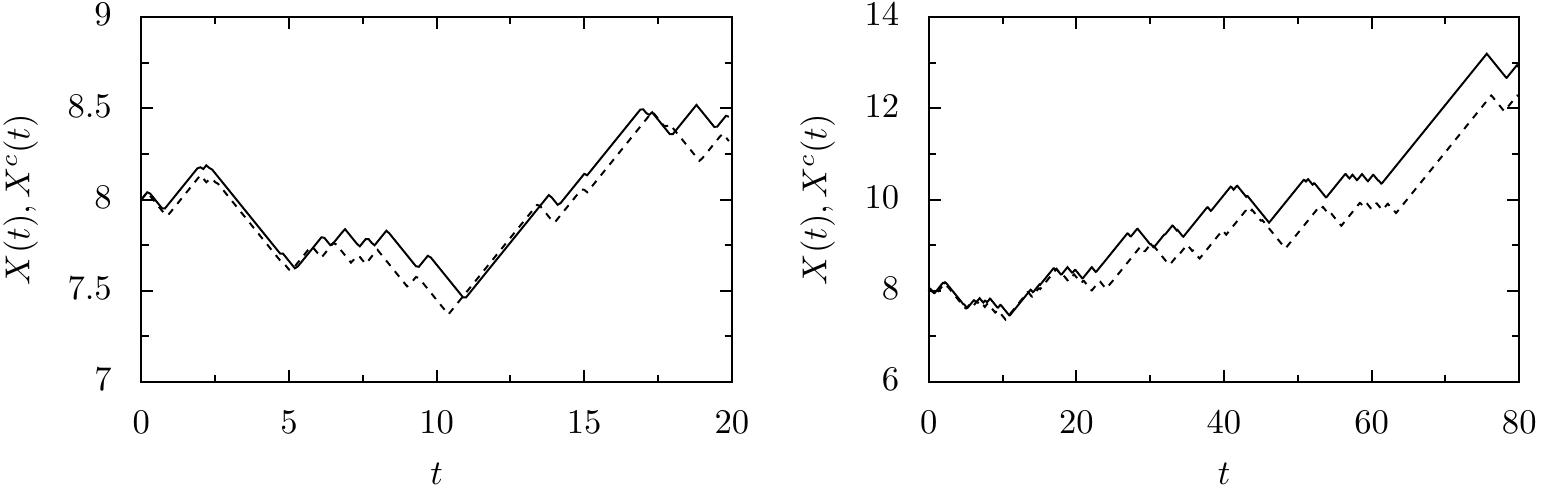}
	\caption{\label{fig:one-particle}Evolution of a single bacterium evolving according to the fine-scale process (\ref{eq:process_noscale}) (solid line) and the corresponding control process (\ref{eq:cprocess}) (dashed) on short (left) and long (right) time-scales.}
\end{center}
\end{figure}
We see a very good coupling initially, which degrades over time. Note the time shift in the short time picture. In the long time picture, coupling is completely lost at some point. 

\paragraph{Expectation and variance as a function of $\epsilon$ and $\bt$.} Next, we repeat the experiment using $N=10000$ particles and compute the empirical mean and variance of the coupling, i.e.,\ 
\[
E\left(\left|X_{\bt}-X_{\bt}^c\right|\right)=\frac{1}{N}\sum_i^N \left(\left|X^i_{\bt}-X^{i,c}_{\bt}\right|\right), \;\;\; \text{resp., }  \;\;\; E\left((X_{\bt}-X_{\bt}^c)^2\right)=\frac{1}{N}\sum_i^N \left(X^i_{\bt}-X^{i,c}_{\bt}\right)^2.
\]
As fine-scale parameters, we choose $\tau=1$, $\lambda_0=b=1$ and several values of $\eps$. The chemoattractant concentration is again given as \eqref{eq:chemoattractant}, now with $\alpha=\beta=1$, $\xi=7.5$ and $\eta=12.5$.  The time interval for the computation is $t\in [0,30/\eps^2]$.
\begin{figure}
	\begin{center}
	\includegraphics[scale=0.68]{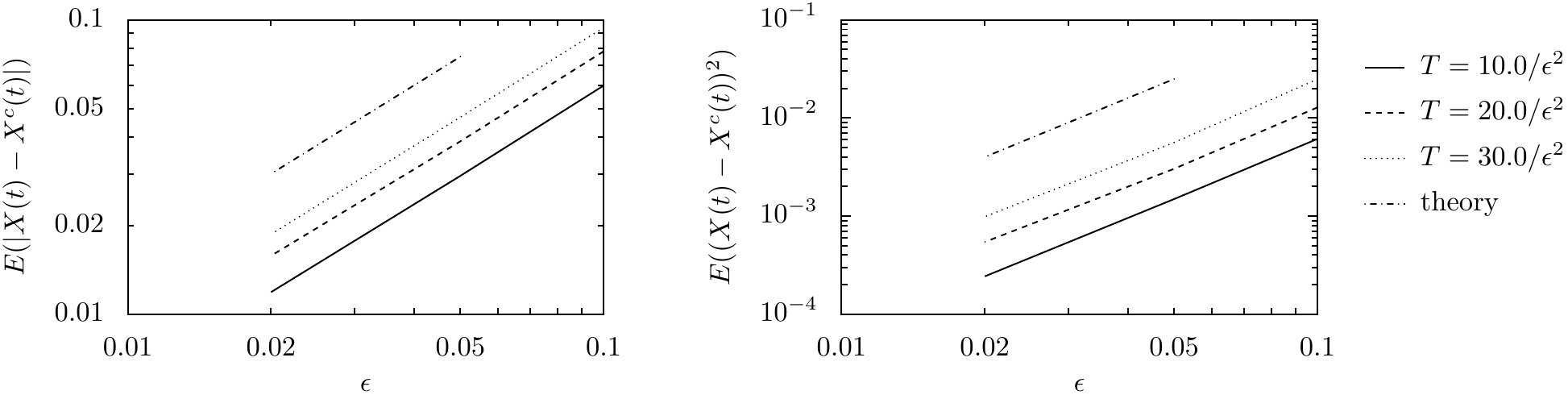}
	\caption{\label{fig:eps-long} Empirical mean (left) and variance (right) of the difference between the fine-scale process (\ref{eq:process_noscale}) and the corresponding control process (\ref{eq:cprocess}) as a function of $\eps$ for different values of the reporting time $T$.  The theoretical slope is indicated with a dashdotted line. The sample size is $N=10000$.}
\end{center}
\end{figure}
Figure \ref{fig:eps-long} shows the dependence in $\eps$ of the coupling, by plotting the empirical mean and variance defined above as a function of $\eps$ for different values of the reporting time $T$. The results shown in figure \ref{fig:eps-long} are in clear accordance with the theoretical slope predicted by the asymptotic analysis.
\begin{figure}
	\begin{center}
	\includegraphics[scale=0.68]{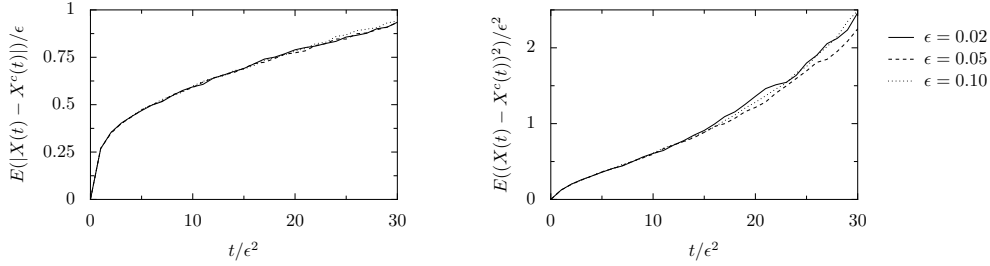}
	\caption{\label{fig:T-long}Evolution of the empirical mean (left) and variance (right) of the difference between the fine-scale process (\ref{eq:process_noscale}) and the corresponding control process (\ref{eq:cprocess}) as a function of $\bt=t/\eps^2$ for different values of $\eps$. The sample size is $N=10000$. }
\end{center}
\end{figure} 
Figure \ref{fig:T-long} shows the mean and variance of the coupling difference as a function of time.  The time dependence of the variance has not been analyzed mathematically in Section~\ref{sec:variance}; the specific behaviour viewed in figure \ref{fig:T-long} and is probably due to: (i) sufficiently short diffusive times; (ii) the specific double-well form of the chemoattractant potential.

\paragraph{Expectation and variance as a function of $\tau$.} Finally, in a last experiment, we illustrate the dependence on $\tau$. To this end, we again simulate $N=10000$ particles choosing $\lambda_0=1$, $\eps=0.1$ and $X_0=7.5$, for different values of $\tau$. The results in figure \ref{fig:tau-long} show that the variance quickly increases with $\tau$ until it reaches a plateau for $\tau>1$. 
\begin{figure}
	\begin{center}
	\includegraphics[scale=0.68]{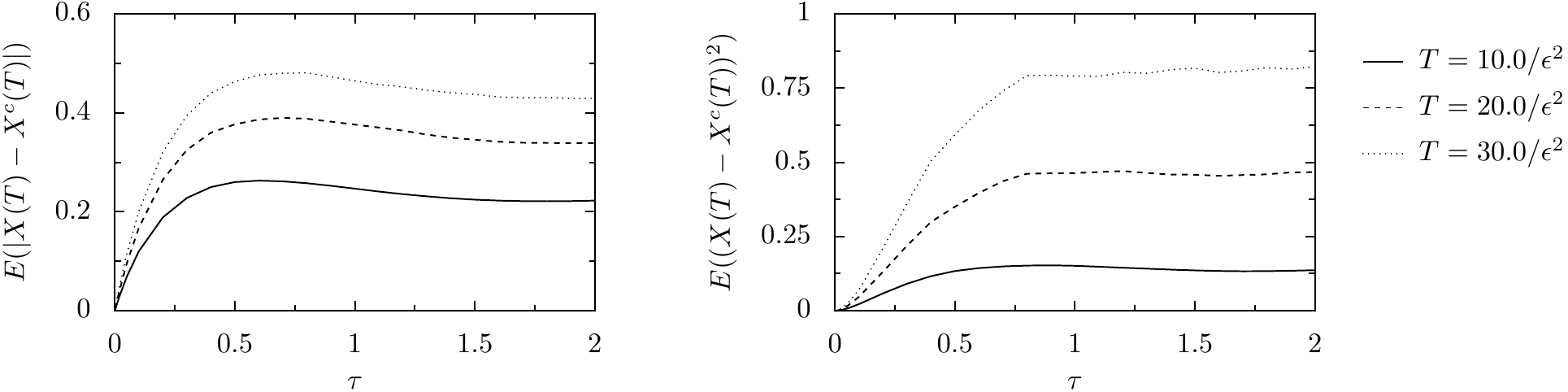}
	\caption{\label{fig:tau-long} Empirical mean (left) and variance (right) of the difference between the fine-scale process (\ref{eq:process_noscale}) and the corresponding control process (\ref{eq:cprocess}) as a function of $\tau$ for different values of the reporting time $T$. The sample size is $N=10000$.}
\end{center}
\end{figure}

\section{Simulation on diffusive time scales\label{sec:appl}}

In this section, we consider a simulation of the density of an ensemble of particles, with and without variance reduction and/or reinitialization, as described in Section~\ref{sec:coupl_princ}. We again restrict ourselves to one space dimension, with domain $x\in [0,20]$ and periodic boundary conditions. In this case, the kinetic equation corresponding to the control process reduces to the system 
\begin{equation} \label{eq:cprocess_hydro1d}
\system{
	 & \partial_t  p^c_++\epsilon\partial_x p^c_+=-\frac{\lambda^c(x,+1)}{2}p_+^c+ \frac{\lambda^c(x,-1)}{2}p_-^c &\\
	 & \partial_t  p^c_--\epsilon\partial_x p^c_-=\frac{\lambda^c(x,+1)}{2}p_+^c -\frac{\lambda^c(x,-1)}{2}p_-^c &  
	 }.
 \end{equation}
of two PDEs, which is straightforward to simulate using finite differences.

We fix the chemoattractant concentration field as \eqref{eq:chemoattractant}, with parameters $\alpha=2$, $\beta=1$, $\xi=7.5$ and $\eta=12.5$. For the internal dynamics, the same model (\eqref{e:scalar-y}-\eqref{eq:rate}, \eqref{eq:control_rate}-\eqref{eq:control_field}) is used. The parameters are $\eps=0.5$, $\lambda_0=1$, $\tau=1$, $\delta t=0.1$. 

All simulations are performed with $N=5000$ particles. The initial positions are uniformly distributed in the interval $x\in[13,15]$; the initial velocities are chosen uniformly, i.e., each particle has an equal probability of having an initial velocity of $\pm\eps$. The initial condition for the internal variable is chosen to be in local equilibrium, i.e., $Y_0^i=S(X^i_0)$. The initial positions and velocities of the control particles are chosen to be identical. 

We discretize the continuum description \eqref{eq:cprocess_hydro1d} on a mesh with $\Delta x = 0.1$ using a third-order upwind-biased scheme, and perform time integration using the standard fourth order Runge--Kutta method with time step $\delta t_{pde}= 10^{-1}$. The initial condition is given as 
\begin{equation}
	p^+(x,0)=p^-(x,0)=\begin{cases}
	0.25 , & {x\in [13,15],} \\
	0, & \text{otherwise.}
\end{cases}
\end{equation}

\paragraph{Simulation without variance reduction.} First, we simulate both stochastic processes up to time $\bar{t}=50$ ($t=50/\eps^2$) and estimate the density of each of these processes $\hat{n}_N(x,\bt)$, resp.~$\hat{n}^c_N(x,\bt)$, without variance reduction.  The density is obtained via binning in a histogram, in which the grid points of the deterministic simulation are the centers of the bins. Figure \ref{fig:no-variance-reduction} (left)
\begin{figure}
	\begin{center}
	\includegraphics[scale=0.8]{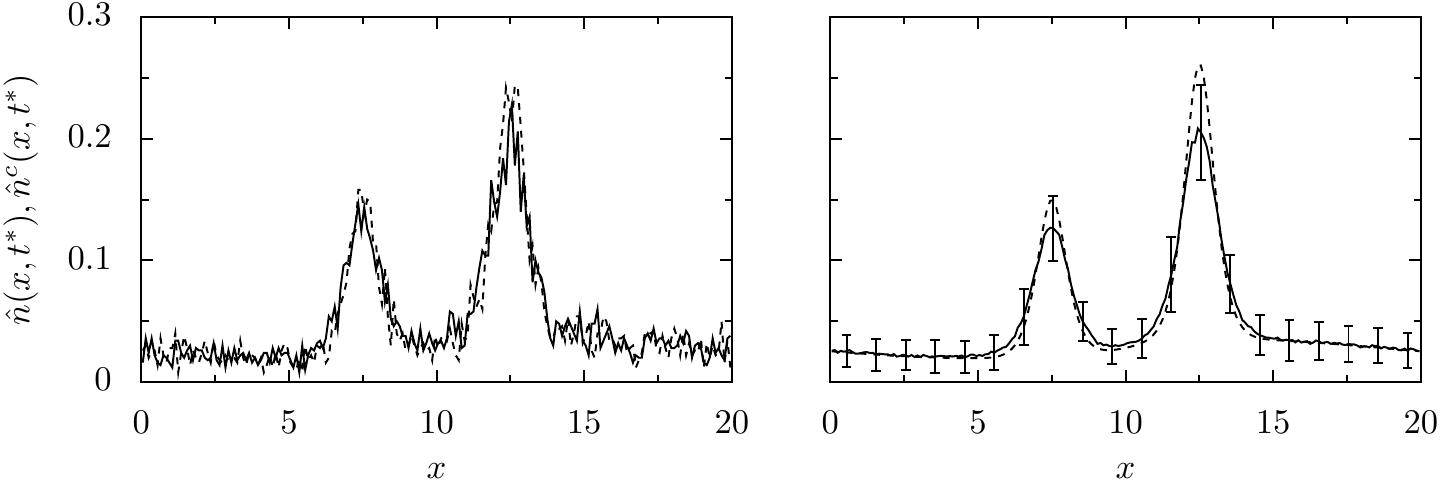}
	\caption{\label{fig:no-variance-reduction} Bacterial density as a function of space at $t=50/\eps^2$ without variance reduction. Left: one realization. Right: mean over $100$ realizations and $95\%$ confidence interval. The solid line is the estimated density from a particle simulation using the process with internal state; the dashed line is estimated from a particle simulation using the control process. Both used $N=5000$ particles. The dotted line is the solution of the deterministic PDE (\ref{eq:cdensity}).   }
\end{center}
\end{figure}
shows the results for a single realization.  We see that, given the fluctuations on the obtained density, it is impossible to conclude on differences between the two models.  This observation is confirmed by computing the average density of both processes over $100$ realizations.  The mean densities are shown in figure \ref{fig:no-variance-reduction} (right), which also reveals that the mean density of the control process is within the $95\%$ confidence interval of the process with internal state.  Both figures also show the density that is computed using the continuum description, which coincides with the mean of the density of the control particles. 

\paragraph{Simulation with variance reduction.} Next, we compare the variance reduced estimation \eqref{eq:var_reduced_estimation} with the density of the control PDE. We reinitialize the control particles after each coarse-scale step, i.e., each $k$ steps of the particle scheme, where $k \delta t = \delta t_{pde}$, (here $k=1$). The results are shown in figure \ref{fig:var-red-reinit}.
\begin{figure}
	\begin{center}
	\includegraphics[scale=0.8]{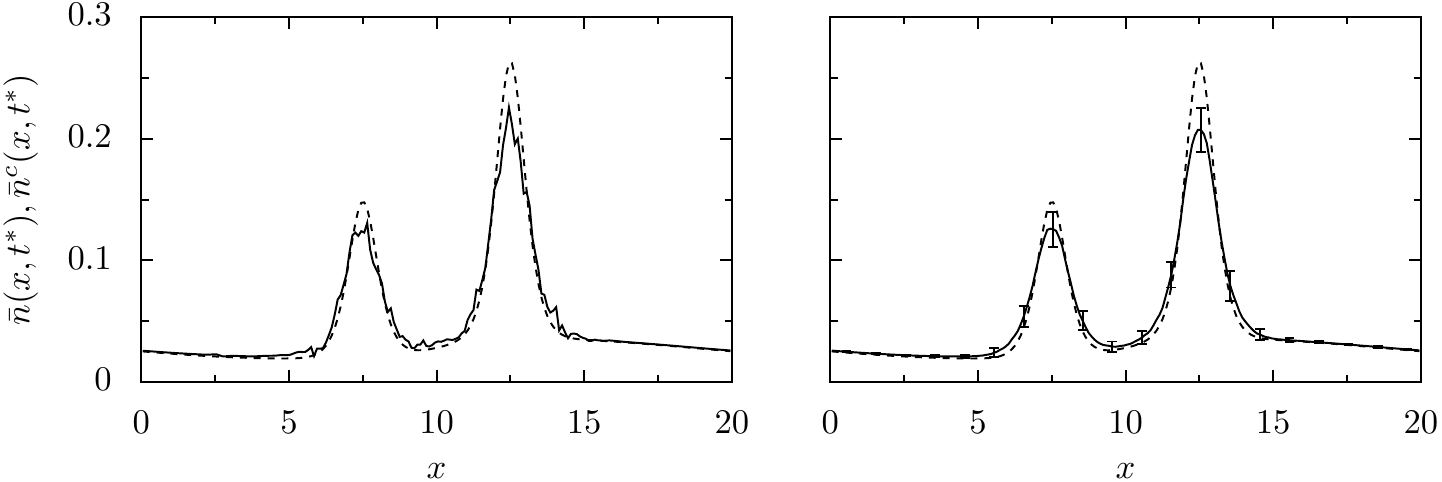}
	\caption{\label{fig:var-red-reinit} Bacterial density as a function of space at $t=50/\eps^2$ with variance reduction and reinitialization. Left: variance reduced density estimation of one realization with $N=5000$ particles (solid) and deterministic solution for the control process (\ref{eq:cdensity}) (dashed). Right: mean over $100$ realization and $95\%$ confidence interval (solid) and the deterministic solution for the control process (\ref{eq:cdensity}) (dashed).}
\end{center}
\end{figure}
We see that, using this reinitialization, the difference between the behaviour of the two processes is visually clear from one realization (left figure).  Also, the resulting variance is such that the density of the control PDE is no longer within the $95\%$ confidence interval of the variance reduced density estimation (right figure). We see that there is a significant difference between both models: the density corresponding to the control process is more peaked, indicating that bacteria that follow the control process are more sensitive to sudden changes in chemoattractant gradient. This difference can be interpreted from the fact that the bacteria with internal state do not adjust themselves instantaneously to their environment, but instead with a time constant $\tau$.

\paragraph{Simulation with variance reduction but without reinitialization. } We also compare the variance reduced estimation \eqref{eq:var_reduced_estimation} with the density of the control PDE without performing any reinitialization of the control process to restore the coupling.
\begin{figure}
	\begin{center}
	\includegraphics[scale=0.8]{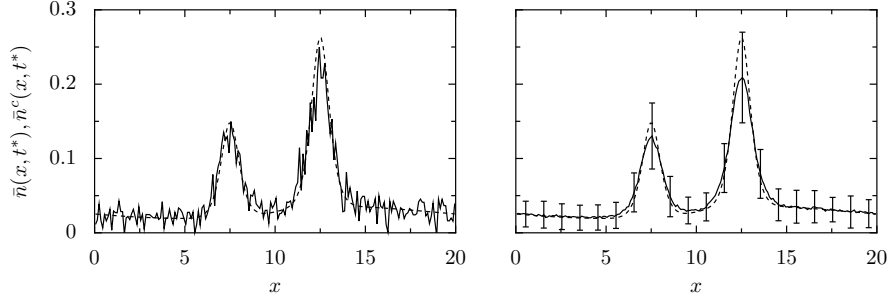}
	\caption{\label{fig:var-red-noreinit} Bacterial density as a function of space at $t=50/\eps^2$ with variance reduction. Left: variance reduced density estimation of one realization with $N=5000$ particles (solid) and deterministic solution for the control process (\ref{eq:cdensity}) (dashed). Right: mean over $200$ realization and $95\%$ confidence interval (solid) and the deterministic solution for the control process (\ref{eq:cdensity}) (dashed).}
\end{center}
\end{figure}
Figure \ref{fig:var-red-noreinit} (right) shows the mean estimation of the density over $200$ realizations, as well as the density of the control process. As evidenced by the error bars, simulating a single realization of the process, even with variance reduction, is not able to reliably reveal this difference. This phenomenon is due to the degeneracy of the coupling on long diffusive times. This is also illustrated in figure \ref{fig:var-red-noreinit} (left), which compares the variance reduced density estimation of a single realization with the density of the control PDE. Note that, as predicted by the analysis, the variance is much larger in regions where $\nabla S(x)$ is large.

Finally, we repeat the experiment with a larger chemoattractant gradient. We again choose the chemoattractant field as \eqref{eq:chemoattractant}, with the same parameters as above, except that we now take $\alpha=5$. We also use the same discretization parameters as above. The variance reduced estimation \eqref{eq:var_reduced_estimation}, as well as the density of the control PDE, are shown in figure \ref{fig:var-red-reinit-alpha5}.
\begin{figure}
	\begin{center}
	\includegraphics[scale=0.8]{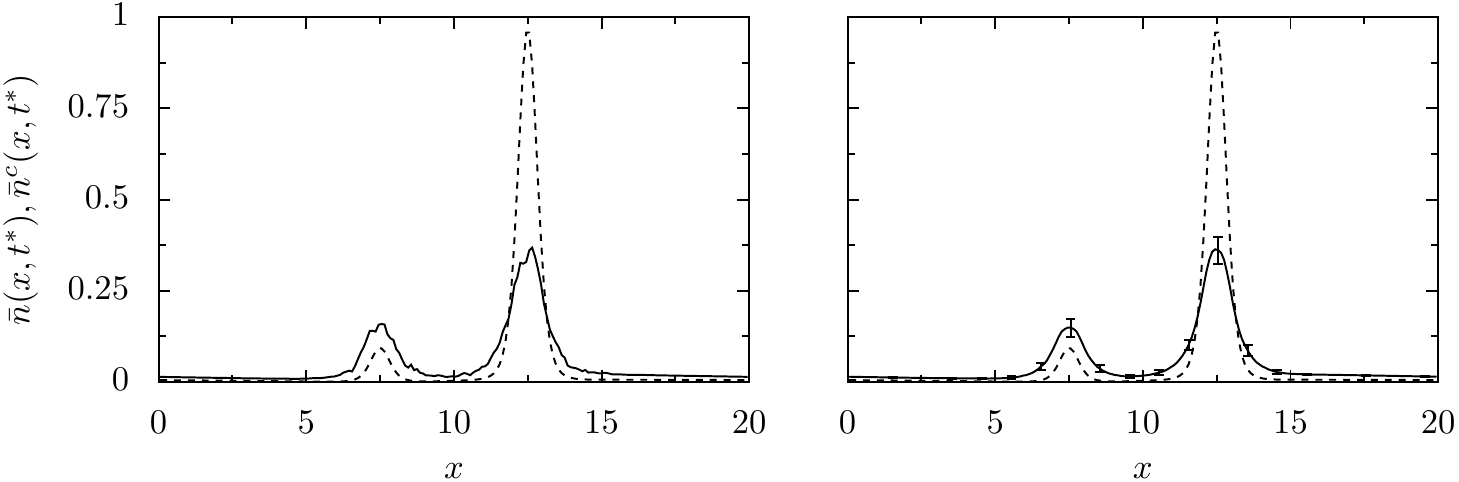}
	\caption{\label{fig:var-red-reinit-alpha5} Bacterial density as a function of space at $t=50/\eps^2$ with variance reduction and reinitialization. Left: variance reduced density estimation of one realization with $N=5000$ particles (solid) and deterministic solution for the control process (\ref{eq:cdensity}) (dashed). Chemoattractant given by \eqref{eq:chemoattractant} with $\alpha=5$. Right: mean over $100$ realization and $95\%$ confidence interval (solid) and the deterministic solution for the control process (\ref{eq:cdensity}) (dashed).}
\end{center}
\end{figure}
We see that, although the difference between the process with internal state and the control process becomes much larger, the variance is still very well controlled by the coupling.

\section{Conclusions and discussion\label{sec:concl}}

In this paper, we studied the simulation of stochastic individual-based models for chemotaxis of bacteria with internal state. We have used a coupling with a simpler, direct gradient sensing model with the same diffusion limit to obtain an ``asymptotic'' variance reduction.  Throughout this work, we have assumed that the computation of the bacterial density using the kinetic description of the direct gradient sensing model can be performed accurately with a grid-based method. Note that in higher spatial dimensions, this may become cumbersome, and the grid-based computation should be carried out at the level of the advection-diffusion description (or any desired moment system) only, using a second level of coupling between the velocity-jump process associated with the kinetic description, and the stochastic differential equation associated with the advection-diffusion limit. This is left for future work.

The current algorithm allows to explore the differences between the fine-scale process with internal state and the simpler, coarse process \emph{using a fixed number of particles, independently of the small parameter $\eps$}.  However, the computational cost to simulated an individual particle over diffusive time-scales becomes very large for $\eps\to 0$.  In future work, we will therefore also study the development of truly ``asymptotic preserving'' schemes in the diffusion asymptotics, in the sense that the computational cost of the simulation of processes is independent of the small parameter $\eps$. 
This will require to deal with two kinds of difficulties: (i)
 We will need to use an asymptotic preserving method to solve the density evolution of the control model on the grid; and (ii) We will need to extrapolate forward in time the state of the fine-scale simulation. 

The first difficulty implies that, instead of solving the full kinetic equation associated with the control process, we only solve its diffusion limit. This may imply the use of a second level of variance reduction, coupling the control velocity jump process and its limiting drift-diffusion process.

The second difficulty, extrapolation in time, is related to the equation-free \cite{KevrGearHymKevrRunTheo03,Kevrekidis:2009p7484} and HMM \cite{EEng03,E:2007p3747} types of methodologies; ideas of this type can be traced back to Erhenfest \cite{Ehrenfest:1990p9192}. One approach is to use a \emph{coarse projective integration} method \cite{GearKevrTheo02,KevrGearHymKevrRunTheo03}.  One then extrapolates the bacterial density over a projective time step on diffusive time scales, after which the projected density needs to be \emph{lifted} to an ensemble of individual bacteria.  For such methods, besides the effect of extrapolation on the variance of the obtained results, also the effects of reconstructing the velocities and internal variables have to be systematically studied. 

\section*{Acknowledgements}

The authors thank Radek Erban, Thierry Goudon, Yannis Kevrekidis and Tony Leli\`evre for interesting discussions that eventually led to this work.  
This work was performed during a research stay of GS at SIMPAF (INRIA - Lille). GS warmly thanks the whole SIMPAF team for its hospitality. GS is a Postdoctoral Fellow of the 
Research Foundation -- Flanders. 
This work was partially supported by the Research Foundation -- Flanders
through Research Project G.0130.03 and by the Interuniversity
Attraction Poles Programme of the Belgian Science Policy Office through grant
IUAP/V/22 (GS). The scientific responsibility rests with its authors.

\bibliographystyle{plain}
%\bibliography{bibliography}
\bibliography{bib-papers}

\appendix

\section{Technical lemmas, used in proof of Theorem~\ref{thm:coupling}}
%\todo[inline]{I have not checked the appendices for typo's ...}

\begin{lem}[Technical lemma]\label{lem:tech1}
  Let $(\theta_n)_{n\geq 1}$ a sequence of independent exponential random variables with mean $1$, and $(T_{n})_{n \geq 1}$ a strictly increasing sequence of random times with $T_0=0$ such that:
  \begin{itemize}
  \item For any $m \geq 1$, the sequence $(\theta_n)_{n\geq m+1}$ is independent of the past $(\theta_n,T_n)_{n\leq m}$.
    \item There is a constant $C$ such that for all $n\geq 0 $, 
      \begin{equation}
        \label{eq:h1}
        \frac{1}{C} \theta_{n+1} \leq T_{n+1} - T_n \leq C \theta_{n+1} .
      \end{equation}
  \end{itemize}
Let $N_t$ a random integer such that:
\[
T_{N_t} \leq t \leq T_{N_t +1}.
\]
Then for any integer $ k \geq 0$, there is a constant $C_k$ independent of $t$ such that:
\begin{equation*}
  \E \pare{ \theta_{N_t+1}  }^k \leq C_k.
\end{equation*}
%\todo[inline]{Where does the superscript $k$ from ? DONE}
\end{lem}
\begin{proof}
  The result may seem rather intuitive, but the proof is slightly tricky. First, we consider sub-intervals of $[0,t]$ of the form:
  \begin{equation*}
    I_p = [\frac{p-1}{P} t, \frac{p}{P} t],
  \end{equation*}
for some $P\geq 1$ and $p = 1, \dots , P$. 

Step (i). The first step consists in proving the following estimate:
\begin{equation}\label{eq:app1}
  \E \pare{ \theta_{N_t+1} ^k } \leq \sum_{p=0}^{P-1} \E\pare{\theta^k \one \pare{ \theta \geq \frac{p}{C} \frac{t}{P}  } } \sup_{p=1 \dots P} \sum_{n=0}^{+\infty} \Proba (T_n \in I_p)
\end{equation}
Indeed,
\begin{align*}
  \E \pare{ \theta_{N_t+1} ^k } & = \sum_{n=0}^{+ \infty} \E \pare{ \theta^k_{n+1} \one \pare{ T_{n} \leq t \leq T_{n +1} } } \\
 & \leq  \sum_{n=0}^{+ \infty} \sum_{p=1}^P \E \pare{ \theta^k_{n+1} \one \pare{ T_{n +1} -T_n \geq \frac{P-p}{P} t  \vert T_n \in I_p } }\Proba \pare{T_n \in I_p}.
\end{align*}
Using the independence of $\theta_{n+1}$ with $T_n$, and the assumption~\eqref{eq:h1}, it yields,
\begin{equation*}
  \E \pare{ \theta^k_{n+1} \one \pare{ T_{n +1} -T_n \geq \frac{P-p}{P} t \big \vert T_n \in I_p } }\leq \E \pare{ \theta^k \one_{ \pare{ \theta \geq \frac{P-p}{C P} t } } }.
\end{equation*}
and finally changing the sum index $p$ to $P-p+1$ yields~\eqref{eq:app1}. 

Step (ii). The second step consists in the following decomposition. 
Let $I=[a,b]$ be a finite interval of $\R^+$, $T_{-1}=-\infty$. Then,
\begin{equation}
  \label{eq:app2}
  \sum_{n=0}^{+\infty} \Proba (T_n \in I) = \sum_{n=0}^{+\infty} \pare{ \Proba(  T_n \in I, T_{n-1} \notin I  ) \sum_{m=0}^{+\infty} \Proba( T_{n+1},\dots,T_{n+m} \in I \vert T_n \in I, T_{n-1} \notin I  )  } .
\end{equation}
The key is to write
\begin{align*}
  \underbrace{\sum_{n=0}^{+\infty} \Proba (T_n \in I)}_{(a_0)} & = \sum_{n=0}^{+\infty} \Proba (T_{n-1} \notin I,T_n \in I) + \sum_{n=0}^{+\infty} \Proba (T_{n-1},T_n \in I) \\
& = \sum_{n=0}^{+\infty} \Proba (T_{n-1} \notin I,T_n \in I) + \underbrace{\sum_{n=0}^{+\infty} \Proba (T_{n} ,T_{n+1} \in I)}_{(a_1)} .
\end{align*}
We wish to iterate the decomposition of $(a_0)$ to $(a_1)$, $(a_2)$, etc... For this purpose, remark that
\begin{align*}
  (a_m) &= \sum_{n=0}^{+\infty} \Proba (T_{n}, \dots ,T_{n+m} \in I) \\
& = \sum_{n=0}^{+\infty} \Proba (T_{n} \in I ) \Proba (T_{n+1}, \dots ,T_{n+m} \in I \vert T_{n} \in I )  \\
& \leq \pare{ \sum_{n=0}^{+\infty} \Proba (T_{n} \in I ) } \Proba ( \theta_1+ \dots + \theta_m \leq C (b-a) ) \\
& \xrightarrow{m \to + \infty} 0 ,
\end{align*}
so that we get in the end
\begin{equation*}
(a_0) = \sum_{n=0}^{+\infty} \sum_{m=0}^{+\infty} \Proba (T_{n-1} \notin I,T_n, \dots T_{n+m} \in I).
\end{equation*}
Factorizing $\Proba(  T_n \in I, T_{n-1} \notin I  ) $ with Bayes' formula yields~\eqref{eq:app2}.

Step (iii). The estimate~\eqref{eq:app1} yields
\begin{equation}
  \label{eq:app4}
  \E \pare{ \theta_{N_t+1} ^k } \leq \sum_{p=0}^{+\infty} \E\pare{\theta^k \one_{ \pare{ \theta \geq \frac{p}{C} \frac{t}{P}  } } } \sum_{m=0}^{+\infty} \Proba\pare{\theta_1+\dots +\theta_{m} \leq \frac{t}{CP}  }.
\end{equation}
Indeed for $I=[a,b]$, using the independence of $(\theta_{n+1}, \dots,\theta_{n+m})$ from $(T_n,T_{n-1})$:
\[
\sum_{m=0}^{+\infty} \Proba( T_{n+1},\dots,T_{n+m} \in I \vert T_n \in I, T_{n-1} \notin I  )   \leq \sum_{m=0}^{+\infty} \Proba\pare{ \theta_{1} + \dots + \theta_{m} \leq C(b-a) },
\]
and
\[
\sum_{n=0}^{+\infty} \Proba(  T_n \in I, T_{n-1} \notin I  ) = \Proba(\exists n \geq 0 \vert T_n \in I  ) \leq 1.
\]
The two last estimates in~\eqref{eq:app2} yields:
\[
\sup_{p=1 \dots P} \sum_{n=0}^{+\infty} \Proba (T_n \in I_p) \leq \sum_{m=0}^{+\infty} \Proba\pare{\theta_1+\dots +\theta_{m} \leq \frac{t}{CP}  },
\]
and~\eqref{eq:app4} follows. Finally, we choose $P= \lfloor t +1\rfloor$ so that the right hand side of~\eqref{eq:app4} does not depend on $t$ any more. Finally we use the exponential convergence of the terms of the two series in the right hand side of~\eqref{eq:app4} (using for instnace large deviation estimates for the second) one, to conclude that the two series indeed converge, and conclude the proof.
\end{proof}

\begin{lem}[Technical lemma]\label{lem:tech2}
  Let $\pare{\theta_n}_{n \geq 0}$ be a sequence of exponential random variables of mean $1$, and $t \mapsto N_t \in \mathbb N$ the Poisson process uniquely defined by:
\[
\sum_{n=1}^{N_t} \theta_n \leq t < \sum_{n=1}^{N_t +1} \theta_n.
\]
Let $n_0 \in \mathbb N$ such that $n_0 \geq 4 \e t$, $p_1 \geq 1$, and $p_2 \geq 1$. Then there is a constant $C_{p_1,p_2}$ independent of $(n_0,t)$ such that:
\[
\E \pare{ \abs{\sum_{n=1}^{N_t+1}\theta_n^{p_1} }^{p_2} \one_{N_t > n_0}} \leq C_{p_1,p_2} 2^{- n_0} .
\]
\end{lem}
\begin{proof}

Step (i). \\
  Let us condition on the different possible values of $N_t > n_0$:
  \begin{align*}
   \E \pare{ \abs{\sum_{n=1}^{N_t+1}\theta_n^{p_1} }^{p_2} \one_{N_t > n_0}} & = \sum_{m=n_0+1}^{+\infty}\E \pare{ \abs{\sum_{n=1}^{N_t+1}\theta_n^{p_1} }^{p_2} \Big \vert N_t = m } \Proba \pare{N_t = m}  \\
& = \sum_{m=n_0+1}^{+\infty}\E \pare{ \abs{\sum_{n=1}^{N_t+1}\theta_n^{p_1} }^{p_2} \Big \vert N_t = m } \e^{-t}\frac{t^{m}}{{m} ! } .
  \end{align*}

Step (ii). \\
 Let us denote $(\Proba_{t,m},\E_{t,m}) = (\Proba\pare{\, . \, \vert N_t = m},\E\pare{\, . \, \vert N_t = m})$ the probability/expectation conditionally on the event $\set{N_t = m}$. Conditionally on $\set{N_t = m}$, the sequence $(\theta_1, \dots,\theta_{m+1})$ has the same distribution as $(t U_{(1)}, t U_{(2)}-t U_{(1)}, \dots, t U_{(m)}- t U_{(m-1)}, t- t U_{(m)}+\theta)$, where $(U_{(1)} < \dots < U_{(m)})$ is the order statistics of $m$ independent random variables $(U_1, \dots , U_m)$ uniformly distributed on the interval $[0,1]$, and $\theta$ is an independent exponential random variable of mean~$1$ (see Section~2.4 of~\cite{Nor98} for classical properties of Poisson processes). Thus denoting $U_{(0)}=0$ we can write:
\begin{align}\label{eq:app11}
 &  \E \pare{ \abs{\sum_{n=1}^{N_t+1}\theta_n^{p_1} }^{p_2} \one_{N_t > n_0}} = \nonumber\\ 
& \qquad \sum_{m=n_0+1}^{+\infty}\E_{t,m} \pare{ \abs{ \pare{t- t U_{(m)}+\theta}^{p_1} + t^{p_1}\sum_{n=1}^{m} \pare{U_{(n)}-U_{(n-1)}}^{p_1} }^{p_2} } \e^{-t}\frac{t^{m}}{{m} ! } .
  \end{align}
Using the inequality $(a+b)^p \leq C_p a^p + b^b$ for any $a,b > 0$, Jensen inequality, and denoting $U_{(m+1)}= 1$, there is a constant $C_{p_1,p_2}$ such that:
\begin{align}
  \label{eq:app12}
 & \abs{ \pare{t- t U_{(m)}+\theta}^{p_1} + t^{p_1}\sum_{n=1}^{m} \pare{U_{(n)}-U_{(n-1)}}^{p_1} }^{p_2} \leq \nonumber \\
& \quad C_{p_1,p2}\pare{\theta^{p_2p_1} + t^{p_1 p_2}(m+1)^{p_2-1} \sum_{n=1}^{m+1} \pare{U_{(n)}-U_{(n-1)}}^{p_1 p_2} }.
\end{align}

Step (iii). \\
Let us compute $\E_{t,m} \pare{ \pare{U_{(n)}-U_{(n-1)}}^{p_1 p_2}  }$ for $n=1, \dots, m+1$. First, let us recall the probability density of a couple $(U_{(i)},U_{(j)})$ for $1 \leq i<j \leq m$ in an order statistics of size $m$:
\begin{align}\label{eq:orderstat}
&  \Proba \pare{U_{(i)}\in [u_i,u_i+d u_i], U_{(j)} \in [u_j,u_j + d u_j] }  \nonumber \\
& \qquad  = m!  \Proba \pare{U_1 < \dots < U_{i-1} < u_{i}} \times \Proba \pare{ U_{i} \in [u_{i},u_{i}+d u_{i}] } \nonumber \\
& \qquad \quad \times \Proba \pare{ u_i < U_{i+1} < \dots <  U_{j-1}<u_{j}} \times \Proba \pare{ U_{j} \in [u_{j},u_{j}+d u_{j}]} \times \Proba \pare{ u_j < U_{j+1} < \dots <  U_m} , \nonumber \\
& \qquad = m! \frac{u_{i}^{i-1}}{(i-1)!} \frac{(u_j-u_i)^{j-i-1}}{(j-i-1)!} \frac{(1-u_j)^{m-j}}{(m-j)!} d u_{i} \, d u_{j} .
\end{align}
As a consequence,
\begin{align*}
  \E_{t,m} \pare{ \pare{U_{(n)}-U_{(n-1)}}^{p_1 p_2}  } & = \int_{0< u_{n-1} < u_n < 1} m! \frac{u_{n-1}^{n-2}}{(n-2)!} (u_n - u_{n-1} )^{p_1 p_2} \frac{(1-u_n)^{m-n}}{(m-n)!} d u_{n-1} \, d u_{n} \\
& = m! (p_1 p_2 ) ! \int_{0< u_{n-1} < u_n < 1} \frac{u_{n-1}^{n-2}}{(n-2)!} \frac{(u_n - u_{n-1} )^{p_1 p_2}}{(p_1 p_2 ) !} \frac{(1-u_n)^{m-n}}{(m-n)!} d u_{n-1} \, d u_{n}  \\
& = \frac{m! (p_1 p_2 ) !}{( m+p_1p_2    )!},
\end{align*}
where in the last line we have used~\eqref{eq:orderstat} with appropriate $(i,j,m)$ to compute the integral. In the same way, the probability density of $U_{(i)}$ for $1 \leq i \leq m$ in an order statistics of size $m$ is given by:
\[
\Proba \pare{U_{(i)}\in [u_i,u_i+d u_i]}=m! \frac{u_{i}^{i-1}}{(i-1)!} \frac{(1-u_i)^{m-i}}{(m-i)!} d u_i,
\]
so that
\begin{align*}
  \E_{t,m} \pare{ U_{(1)} ^{p_1 p_2} } & = m! (p_1 p_2 ) ! \int_{0< u_{1} < 1} \frac{u_{1}^{p_1 p_2}}{(p_1 p_2)!} \frac{(1-u_1)^{m-1}}{(m-1)!} d u_1 \\
 & = \frac{m! (p_1 p_2 ) !}{( m+p_1p_2    )!},
\end{align*}
and in the same way,
\[
 \E_{t,m} \pare{ (1-U_{(m)})^{p_1 p_2} } = \frac{m! (p_1 p_2 ) !}{( m+p_1p_2    )!} .
\]

Step (iv). \\
Remarking that $\E_m(\theta^{p_1p_2})=(p_1p_2)!$, and inserting the computations of Step (ii) in~\eqref{eq:app11}-\eqref{eq:app12} yields:
\begin{align*}
 &  \E \pare{ \abs{\sum_{n=1}^{N_t+1}\theta_n^{p_1} }^{p_2} \one_{N_t > n_0}} \leq \nonumber\\ 
& \qquad C_{p_1p_2}\sum_{m=n_0+1}^{+\infty} \pare{ (p_1p_2)! + t^{p_1 p_2}(m+1)^{p_2-1}  \frac{m! (p_1 p_2 ) !}{( m+p_1p_2    )!}   }\e^{-t}\frac{t^{m}}{{m} ! } .
\end{align*}
Using the inequalities $ \frac{(m+1)^{p_2-1} m !}{( m+p_1p_2    )!} \leq 1$, $\e^{-t}\leq 1$ and
\[
\sum_{m=n_0+1}^{+\infty} \frac{t^{m}}{{m} ! } \leq \frac{t^{n_0}}{ n_0 ! },
\]
it yields 
\[
\E \pare{ \abs{\sum_{n=1}^{N_t+1}\theta_n^{p_1} }^{p_2} \one_{N_t > n_0}}  \leq C_{p_1p_2} \frac{t^{n_0+p_1p_2}}{n_0 !},
\]
so that in the end, using the assumption $n_0 \geq 4 \e t$ and the Stirling formula,
\begin{align*}
 &  \E \pare{ \abs{\sum_{n=1}^{N_t+1}\theta_n^{p_1} }^{p_2} \one_{N_t > n_0}} \leq C_{p_1p_2}   (\frac{n_0}{4 \e })^{n_0+p_1p_2} \frac{1}{\sqrt{n_0}} (\frac{\e }{n_0})^{n_0} \\
& \qquad \leq C_{p_1p_2} \frac{n_0^{p_1 p_2}}{\sqrt{n_0} 2 ^{n_0}} 2^{- n _0} \leq  C_{p_1p_2} 2^{- n _0} ,
\end{align*}
which yields the result.
\end{proof}

\end{document}